\documentclass[11pt]{article}

\usepackage[centertags]{amsmath}
\usepackage{amsfonts}
\usepackage{amssymb}
\usepackage{amsthm}
\usepackage{newlfont}
\usepackage{bibspacing}

\newtheorem{thm}{Theorem}[section]
\newtheorem*{thm*}{Theorem}
\newtheorem{cor}[thm]{Corollary}

\newtheorem{lem}[thm]{Lemma}

\newtheorem{prop}[thm]{Proposition}
\newtheorem*{prop*}{Proposition}

\newtheorem*{conj*}{Conjecture}

\newtheorem*{dfn*}{Definition}
\theoremstyle{definition}
\newtheorem{rem}[thm]{\textbf{Remark}}
\newtheorem*{rmk*}{Remark}
\newtheorem*{fact*}{Fact}
\newtheorem{ex}[thm]{Example}
\theoremstyle{proof}

\numberwithin{equation}{section}

\newcommand{\norm}[1]{\left\Vert#1\right\Vert}

\newcommand{\abs}[1]{\left\vert#1\right\vert}
\newcommand{\set}[1]{\left\{#1\right\}}
\newcommand{\brac}[1]{\left(#1\right)}

\newcommand{\Real}{\mathbb{R}}

\newcommand{\eps}{\varepsilon}

\newcommand{\C}{\mathcal{C}}

\newcommand{\Vol}[1]{\textnormal{Vol} \left(#1 \right)} 

\newtheorem*{namedprop}{\theoremname}
\newcommand{\theoremname}{testing}

\newlength{\defbaselineskip}
\newcommand{\setlinespacing}[1]%
           {\setlength{\baselineskip}{#1 \defbaselineskip}}

\newcommand{\B}{\mathcal{B}}

\def \F {\mathcal{F}}

\newcommand{\diam}{\textnormal{diam}}
\def \Cap {\textnormal{Cap}}
\def \tVol {\textnormal{Vol}}

\begin{document}

\title{On the role of Convexity in Isoperimetry, Spectral Gap and Concentration}

\author{\\
Emanuel Milman\textsuperscript{1} \\ \\}

\footnotetext[1]{School of Mathematics,
Institute for Advanced Study, Einstein Drive, Simonyi Hall, Princeton, NJ 08540.
Email: emilman@math.ias.edu.\\
Supported by NSF under agreement \#DMS-0635607.}

\maketitle

\begin{abstract}
We show that for convex domains in Euclidean space, Cheeger's isoperimetric inequality,
spectral gap of the Neumann Laplacian, exponential concentration of Lipschitz functions, and
the a-priori weakest requirement that Lipschitz functions have \emph{arbitrarily slow} uniform tail-decay,
are all quantitatively equivalent (to within universal constants, independent of the
dimension). This substantially extends previous results of
Maz'ya, Cheeger, Gromov--Milman, Buser and Ledoux. As an application, we conclude a sharp quantitative stability
result for the spectral gap of convex domains under convex perturbations which preserve volume (up to
constants) and under maps which are ``on-average'' Lipschitz. We also provide a new characterization
(up to constants) of the spectral gap of a convex domain, as one over the square of the average distance from the ``worst''
subset having half the measure of the domain. In addition, we easily recover and extend many previously known lower bounds on the spectral gap of convex domains, due to Payne--Weinberger, Li--Yau, Kannan--Lov\'asz--Simonovits, Bobkov and Sodin. The proof involves estimates on the diffusion semi-group following Bakry--Ledoux and a result from Riemannian Geometry on the concavity of the isoperimetric profile. Our results extend to the more general setting of Riemannian
manifolds with density which satisfy the $CD(0,\infty)$ curvature-dimension condition of Bakry-\'Emery.
\end{abstract}

\section{Introduction}

Let $(\Omega,d,\mu)$ denote a metric probability space. More
precisely, we assume that $(\Omega,d)$ is a separable metric space
and that $\mu$ is a Borel probability measure on $(\Omega,d)$ which
is not a unit mass at a point. Although it is not essential for the
ensuing discussion, it will be more convenient to specialize to the
case where $\Omega$ is a smooth complete oriented $n$-dimensional
Riemannian manifold $(M,g)$, $d$ is the induced geodesic distance,
and $\mu$ is an absolutely continuous measure with respect to the
Riemannian volume form $vol_M$ on $M$. A question which goes back at
least to the 19th century (motivating the solution to the
isoperimetric problem in $\Real^n$), and arguably much before that
(e.g. Dido's problem), pertains to the interplay between the metric
$d$ and the measure $\mu$. There are various different ways to
measure this relationship, which may be typically arranged according
to strength, forming a hierarchy.
In this work, we will be primarily concerned with three such different ways.

\subsection{The Hierarchy}

The first way is by means of an isoperimetric inequality. Recall that Minkowski's (exterior)
boundary measure of a Borel set $A \subset \Omega$, which we denote here by $\mu^+(A)$, is defined
as:
\[
 \mu^+(A) := \liminf_{\eps \to 0} \frac{\mu(A_{\eps,d}) - \mu(A)}{\eps}~,
\]
where $A_{\eps,d} := \set{x \in \Omega ; \exists y \in A \;\; d(x,y) < \eps}$ denotes the
$\eps$-neighborhood of $A$ with respect to the metric $d$. It is clear that the boundary measure is
a natural generalization of the notion of surface area to the metric probability space setting. An
isoperimetric inequality measures the relation between $\mu^+(A)$ and $\mu(A)$ by means of the
isoperimetric profile $I = I_{(\Omega,d,\mu)}$, defined as the pointwise maximal function $I : [0,1]
\rightarrow \Real_+$, so that $\mu^+(A) \geq I(\mu(A))$
for all Borel sets $A \subset \Omega$. A set $A$ for which equality above is attained is
called an isoperimetric minimizer. Since $A$ and $\Omega \setminus A$ will typically
(but not necessarily, consider $\mu$ with non-continuous density) have the same boundary measure,
it will be convenient to also define $\tilde{I} = \tilde{I}_{(\Omega,d,\mu)}$ as the function
$\tilde{I} :[0,1/2] \rightarrow \Real_+$ given by $\tilde{I}(t) := \min(I(t),I(1-t))$.

A very useful isoperimetric inequality was considered by Cheeger \cite{CheegerInq}
(and in a more general form, independently by V. G. Maz'ya \cite{MazyaSobolevImbedding,MazyaCapacities}):

\begin{dfn*}
The space $(\Omega,d,\mu)$ is said to satisfy Cheeger's isoperimetric inequality if:
\[
\exists D>0 \; \text{ such that } \; \tilde{I}_{(\Omega,d,\mu)}(t) \geq D t \;\; \forall t \in [0,1/2] ~.
\]
The best possible constant $D$ above is denoted by $D_{Che} = D_{Che}(\Omega,d,\mu)$.
\end{dfn*}

A second way to measure the interplay between $d$ and $\mu$ is given
by functional inequalities. Let $\F = \F(\Omega,d)$ denote the space
of functions which are Lipschitz on every ball in $(\Omega,d)$ - we
will call such functions ``Lipschitz-on-balls'' - and let $f \in
\F$. We will consider functional inequalities which measure the
relation between $\norm{f}_{L_p(\mu)}$ and $\norm{\abs{\nabla
f}}_{L_q(\mu)}$, for $0 < p,q \leq \infty$ (more general Orlicz
norms will be treated in
\cite{EMilmanRoleOfConvexityInFunctionalInqs}). Here, the effect of
the metric $d$ is via the Riemannian metric $g$ which is used to
measure $\abs{\nabla f}:=g(\nabla f,\nabla f)^{1/2}$, although more
general ways exist to define $\abs{\nabla f}$ in the non manifold
setting. Of course if $f$ is constant there is no sense to compare
against $\norm{\abs{\nabla f}}_{L_q(\mu)} = 0$, so we will need to
exclude these cases. To this end, we will require that either the
expectation $E_\mu f$ or median $M_\mu f$ of $f$ are 0. Here $E_\mu
f = \int f d\mu$ and $M_\mu f$ is a value so that $\mu(f \geq M_\mu
f) \geq 1/2$ and $\mu(f \leq M_\mu f) \geq 1/2$.

A well known example of a functional inequality was studied by Poincar\'e:
\begin{dfn*}
The space $(\Omega,d,\mu)$ is said to satisfy Poincar\'e's inequality if:
\[
\exists D>0 \; \text{ such that } \; \forall f \in \F \;\;\;\;  D
\norm{f - E_\mu f}_{L_2(\mu)} \leq \norm{\abs{\nabla f}}_{L_2(\mu)} ~.
\]
The best possible constant $D$ above is denoted by $D_{Poin} = D_{Poin}(\Omega,d,\mu)$.
\end{dfn*}
It is well known (e.g. \cite{FollandBook}) that under appropriate smoothness assumptions, Poincar\'e's inequality is equivalent to the existence of a spectral gap of an appropriate Laplacian operator $-\Delta_{g,\mu}$ on $(M,g)$ associated to the measure $\mu$ with corresponding boundary conditions on its support. When $\mu$ is uniform on a domain $\Omega \subset (M,g)$, $\Delta_{g,\mu}$ coincides with the usual Laplace-Beltrami operator $\Delta_g$ with Neumann boundary conditions on $\Omega$. The first non-trivial eigenvalue of $-\Delta_{g,\mu}$ (the ``spectral gap'') is then precisely $D_{Poin}^2(\Omega,d,\mu)$.

\medskip

A third way to measure the relation between $d$ and $\mu$ is given
by concentration inequalities. These measure how tightly
$1$-Lipschitz functions are concentrated about their mean, by
providing a quantitative estimate on the tail decay $\mu(|f - E_\mu
f| \geq t)$. A typical situation is given by the following example:
\begin{dfn*}
The space $(\Omega,d,\mu)$ is said to have exponential concentration if:
\[
\exists c,D>0 \; \text{ such that } \; \forall \text{ 1-Lipschitz } f \;\; \forall t>0 \;\;\;\;
\mu(|f - E_\mu f| \geq t) \leq c \exp(-D t) ~.
\]
Fixing $c=e$, the best possible constant $D$ above is denoted by $D_{Exp} =
D_{Exp}(\Omega,d,\mu)$. The best constant for a specific $f$ is denoted by $D_{Exp}(f)$.
\end{dfn*}

It is known that the three examples mentioned above are arranged in a hierarchy. It was shown by
Cheeger \cite{CheegerInq}, and in a more general form, independently by Maz'ya
\cite{MazyaSobolevImbedding,MazyaCheegersInq1,MazyaCapacities} (see also \cite{GrigoryanAboutMazya}),
that Cheeger's isoperimetric inequality always implies Poincar\'e's inequality (or spectral gap):
\begin{thm}[Maz'ya, Cheeger] \label{thm:Cheeger}
$D_{Poin} \geq D_{Che}/2$ (``Cheeger's inequality'').
\end{thm}
The fact that Poincar\'e's inequality implies exponential concentration was first shown by M.
Gromov and V. Milman \cite{GromovMilmanLevyFamilies} in the Riemannian setting, and subsequently by
other authors in other settings as well (e.g. \cite{AlonMilmanSpectralGapImpliesConcentration}, see \cite{Ledoux-Book} and the references therein):
\begin{thm}[Gromov--Milman] \label{thm:GM}
There exists a universal numeric constant $c>0$ such that $D_{Exp} \geq c D_{Poin}$.
\end{thm}

\subsection{Reversing the Hierarchy}

It is known and easy to show that these implications
\emph{cannot} be reversed in general.
For instance, using $([-1,1],\abs{\cdot},\mu_\alpha)$ where
$d\mu_\alpha = \frac{1+\alpha}{2} |x|^\alpha dx$ on $[-1,1]$, clearly $\mu_\alpha^+([0,1]) = 0$ so $D_{Che} = 0$, whereas one can show that $D_{Poin}> 0$ for $\alpha \in (0,1)$ using a criterion for the Poincar\'e inequality on $\Real$ due to Artola, Talenti and Tomaselli (cf. Muckenhoupt \cite{MuckenhouptHardyInq}).
In addition, if $\mu$ is supported on a set
$\Omega$ with diameter bounded by a finite $D$, trivially one has
$D_{Exp} \geq 1/D > 0$; but if we choose $\Omega$ to be
disconnected, we will always have $D_{Poin} = D_{Che} = 0$. In fact,
one need not impose such topological obstructions on $\Omega$, it is
also easy to construct a connected set with arbitrarily narrow
``necks''. We conclude that in order to have any chance of
reversing the above implications, we will need to add some
additional assumptions, which will prevent the existence of such
narrow necks. Intuitively, it is clear that some type of convexity
assumptions are a natural candidate.
We start with two important examples when $(M,g) =
(\Real^n,\abs{\cdot})$ and $\abs{\cdot}$ is some
fixed Euclidean norm:
\begin{itemize}
\item
$\Omega$ is an \emph{arbitrary} bounded convex domain in $\Real^n$ ($n \geq 2$), and $\mu$ is the
uniform probability measure on $\Omega$.
\item
$\Omega = \Real^n$ ($n \geq 1$) and $\mu$ is an \emph{arbitrary}
absolutely continuous log-concave probability measure, meaning that
$d\mu = \exp(-\psi) dx$ where $\psi: \Real^n \rightarrow \Real \cup
\set{+\infty}$ is convex (we refer to the paper
\cite{Borell-logconcave} of C. Borell for more information).
\end{itemize}

In both cases, we will say that ``our convexity assumptions are fulfilled''. More generally, we
present the following definition:

\begin{dfn*}
We will say that our \emph{smooth convexity assumptions} are fulfilled if:
\begin{itemize}
\item
$(M,g)$ denotes an $n$-dimensional ($n\geq 2$) smooth complete oriented connected Riemannian manifold
or $(M,g)=(\Real,\abs{\cdot})$, and $\Omega = M$.
\item
$d$ denotes the induced geodesic distance on $(M,g)$.
\item
$d\mu = \exp(-\psi)  dvol_M$,  $\psi \in C^2(M)$, and as tensor fields on $M$:
\begin{equation} \label{eq:Intro-BE}
Ric_g + Hess_g \psi \geq 0 ~.
\end{equation}
\end{itemize}
We will say that our \emph{convexity assumptions} are fulfilled if
$\mu$ can be approximated in total-variation by measures $\set{\mu_m}$ so that $(\Omega,d,\mu_m)$ satisfy our smooth convexity assumptions.
\end{dfn*}

The condition (\ref{eq:Intro-BE}) is the well-known Curvature-Dimension condition $CD(0,\infty)$,
introduced by Bakry and \'Emery in their influential paper \cite{BakryEmery} (in the more abstract
framework of diffusion generators). Here $Ric_g$ denotes the Ricci curvature tensor and $Hess_g$ denotes the second covariant derivative. When the Ricci tensor satisfies a slightly relaxed condition
$Ric_g \geq -K g$, $K \geq 0$, it was first shown by Buser \cite{BuserReverseCheeger} that the
implication in Theorem \ref{thm:Cheeger} can be reversed. We only quote the $K=0$ case, which in our
setting reads:

\begin{thm}[Buser] \label{thm:Buser}
If $\mu$ is uniform on a closed $n$-dimensional manifold $(M,g)$ and $Ric_g
\geq 0$ then $D_{Che} \geq c D_{Poin}$, where $c>0$ is a universal numeric constant.
\end{thm}

The fact that the constant $c$ above does not depend on the
dimension $n$ is quite remarkable. Buser's theorem was recently further generalized by M. Ledoux \cite{LedouxSpectralGapAndGeometry} (following the method developed by Bakry--Ledoux \cite{BakryLedoux}) to the Bakry-\'Emery abstract setting.
Again, we only quote the $CD(0,\infty)$ case:

\begin{thm}[Ledoux] \label{thm:Ledoux}
Under our smooth convexity assumptions $D_{Che} \geq c D_{Poin}$, where $c>0$ is a universal numeric constant.
\end{thm}

\subsection{Main Theorem}

How about reversing the implication in Theorem \ref{thm:GM} under our convexity assumptions?
This is one of the statements in our Main Theorem below. A second statement, which is much more surprising,
concerns a very weak type of concentration inequality, which we introduce:

\begin{dfn*}
The space $(\Omega,d,\mu)$ is said to satisfy First-Moment concentration if:
\begin{equation} \label{eq:FM-def}
\exists D>0 \; \text{ such that } \; \forall \text{ 1-Lipschitz } f \;\;\;\;
\norm{f - E_\mu f}_{L_1(\mu)} \leq \frac{1}{D} ~.
\end{equation}
The best possible constant $D$ above is denoted by $D_{FM} = D_{FM}(\Omega,d,\mu)$.
\end{dfn*}

Clearly, by the Markov-Chebyshev inequality, First-Moment concentration implies \emph{linear}
tail-decay:
\[
 \forall \text{ 1-Lipschitz } f \;\; \forall t>0 \;\;\;\; \mu(|f - E_\mu f| \geq t) \leq
\frac{1}{D_{FM} t} ~,
\]
and decay slightly faster than linear implies (integrating by parts) First-Moment concentration.
The First-Moment concentration is clearly a-priori \emph{much weaker} than exponential concentration. Our Main
Theorem, first announced in \cite{EMilman-RoleOfConvexityCRAS}, asserts that under our convexity assumptions, not only is First-Moment concentration \emph{equivalent} to exponential concentration, but in fact also to the a-priori stronger
inequalities of Poincar\'e and Cheeger:

\begin{thm} \label{thm:Main}
Under our convexity assumptions, the following statements are equivalent:
\begin{enumerate}
\item Cheeger's isoperimetric inequality (with $D_{Che}$).
\item Poincar\'e's inequality (with $D_{Poin}$).
\item Exponential concentration inequality (with $D_{Exp}$).
\item First Moment concentration inequality (with $D_{FM}$).
\end{enumerate}
The equivalence is in the sense that the constants above satisfy $D_{Che} \simeq D_{Poin} \simeq D_{Exp} \simeq D_{FM}$.
\end{thm}

Here and below, $A \simeq B$ means that $C_1 B \leq A \leq C_2 B$,
with $C_i>0$ some universal numerical constants, independent of any
other parameter, and in particular the dimension $n$. We will see in
Section \ref{sec:extensions} that the use of the First-Moment is not
essential in Statement (4); we may have required any \emph{arbitrarily slow} uniform tail decay, instead of linear decay.
In other words, if:
\begin{equation} \label{eq:slow-tail-decay}
 \exists \alpha: \Real_+ \rightarrow [0,1] \;\;\; \alpha(t) \rightarrow_{t \rightarrow \infty} 0\;\;\;\; \forall \text{ 1-Lipschitz } f \;\;\forall t>0 \;\;\;\; \mu(|f - E_\mu f| \geq t) \leq \alpha(t) ~,
\end{equation}
where $\alpha$ decays to 0 \emph{arbitrarily slow}, we can deduce under our convexity assumptions that Lipschitz
functions have in fact much faster \emph{exponential} tail decay (with rate depending solely on $\alpha$), and in
addition the stronger inequalities of Poincar\'e and Cheeger, as above.
In this sense, our result
extends the well-known Kahane-Khinchine type inequalities in Convexity Theory
(e.g. consequences of Borell's Lemma \cite{Borell-logconcave}, see \cite{Milman-Schechtman-Book} for an overview)
stating that \emph{linear functionals} have comparable moments, ensuring exponential tail decay,
to the same statement for the ``worst'' $1$-Lipschitz function (see Remark \ref{rem:FM-Exp-Equivalent}).

\medskip

The Main Theorem may also be interpreted as stating that under our
convexity assumptions, there exists a single $1$-Lipschitz function
$f$ whose level sets \emph{on average} attain the minimum (up to
constants) in Cheeger's isoperimetric inequality (see Section \ref{sec:extensions}).
In fact, one may choose this function to be of the form $f(x) = d(x,A)$, where $A$ is some set with $\mu(A)\geq
1/2$.
This is expressed in the following reformulation of the Main Theorem: \begin{thm} \label{thm:Lip-form}
Under our convexity assumptions on $(\Omega,d,\mu)$:
\[
 D_{Che}(\Omega,d,\mu) \simeq \inf \set{ \frac{1}{\int_\Omega d(x,A) d\mu} \; ; \; A \subset
\Omega \;,\; \mu(A) \geq 1/2} ~.
\]
\end{thm}
\noindent
Equivalently, this is tantamount to saying that under our convexity assumptions, it is only necessary
to use test functions of the form $f(x)=d(x,A)$ when testing (up to a universal numeric constant) for the
spectral gap $D_{Poin}^2$ in Poincar\'e's inequality. Clearly, without any further assumptions, all of the above statements are in general false.

\subsection{Applications to Spectral Gap of Convex Domains}

In Section \ref{sec:cor}, we deduce from our Main Theorem \ref{thm:Main} several new results pertaining to the spectral gap of convex domains, and recover and extend numerous previously known results as well.
We will formulate our results in
Euclidean space $(\Real^n,\abs{\cdot})$, even though they hold for
the most part under our more general convexity assumptions.

For a bounded domain $\Omega \subset (\Real^n,\abs{\cdot})$, let $\lambda_\Omega$ denote the uniform
probability measure on $\Omega$, and denote $D_{Poin}(\Omega) :=
D_{Poin}(\Omega,\abs{\cdot},\lambda_\Omega)$. As our main application, we deduce the following
stability result for the spectral gap $D_{Poin}^2(\Omega)$ of the Neumann Laplacian on $\Omega$ under perturbations of the domain $\Omega$. Clearly, there can be no stability result without some further assumptions, which we add in the form of convexity. We formulate the stability in terms of the Cheeger constant $D_{Che}(\Omega) :=
D_{Che}(\Omega,\abs{\cdot},\lambda_\Omega)$ (this is a-priori stronger than using $D_{Poin}(\Omega)$ by the Maz'ya--Cheeger inequality, but in fact equivalent in the class of convex domains by the Buser-Ledoux Theorems):

\begin{thm} \label{thm:Intro-stability}
Let $K,L$ denote two bounded convex domains in $(\Real^n,\abs{\cdot})$. If:
\[
 \Vol{K \cap L} \geq v_K \Vol{K} \;\;\;,\;\;\; \Vol{K \cap L} \geq v_L \Vol{L} ~,
\]
then:
\begin{equation} \label{eq:thm-Intro-stability}
 D_{Che}(K) \geq c \; \frac{v_K^2}{\log(1+1/v_L)} D_{Che}(L) ~,
\end{equation}
where $c>0$ is some universal numeric constant.
\end{thm}
\noindent
Here $\text{Vol}$ denotes the Lebesgue measure. In particular, we see that:
\[
 \Vol{K} \simeq \Vol{L} \simeq \Vol{K \cap L} \;\; \Rightarrow \;\; D_{Che}(K,|\cdot|,\lambda_K) \simeq D_{Che}(L,|\cdot|,\lambda_L) ~.
\]
Note that $K,L$ satisfying the above condition can be very different geometrically (consider for instance a Euclidean ball of radius $1$ and its intersection with a
centered slab of width $10/\sqrt{n}$), and yet share essentially the same spectral gap. Also note that our stability result holds with respect to all possible Euclidean structures $|\cdot|$ simultaneously, since the assumption in the left-hand side above is independent of the Euclidean structure.

We also observe that the quantitative dependence on $v_K,v_L$ in (\ref{eq:thm-Intro-stability}) is essentially best possible:
the logarithmic dependence on $1/v_L$ is (up to numeric constants) optimal, and the quadratic dependence on $v_K$ cannot
be improved beyond linear (and is in fact optimal in some restricted range, see Example \ref{ex:quadratic}).
In addition, Theorem \ref{thm:Intro-stability} implies that when $\frac{1}{a} L \subset K \subset L b$ with $a,b \geq 1$, $a b \leq 1 + \frac{c}{n}$, then $D_{Che}(K) \simeq D_{Che}(L)$. In fact, when $a b \leq 1 + \frac{s}{n}$ with $1 \leq s \leq n$, we obtain in Corollary \ref{cor:isomorphism} the best possible (up to numeric constants) quantitative bounds on $D_{Che}(K) / D_{Che}(L)$ as a function of $s$ (see Example \ref{ex:isomorphism}).
To the best of our knowledge, no quantitative bounds on the stability of $D_{Che}$ for convex domains under convex perturbations were previously known.
Completely analogous stability results hold for log-concave probability measures as well (see Theorem \ref{thm:stability-inclusion-lc}). Another useful result which we deduce from our Main
Theorem is that Cheeger's constant is preserved under maps which are not necessarily Lipschitz, but
rather Lipschitz on average (see Theorem \ref{thm:Lip-on-average}).

\smallskip

An intriguing conjecture of Kannan, Lov\'asz and Simonovits
\cite{KLS} states that under a natural non-degeneracy condition on a
bounded convex domain $K$ in $(\Real^n,\abs{\cdot})$, $D_{Che}(K)
\simeq 1$, independently of the dimension $n$. The upper bound
follows from standard Convexity Theory, but the lower bound is far
from being resolved. There are many known lower bounds which provide
dimension dependent results, and we are able to easily recover many
of them, without appealing to the localization method used by
Kannan--Lov\'asz--Simonovits (which may be traced back to the work
of Gromov--Milman \cite{Gromov-Milman}). These include results by
Payne and Weinberger \cite{PayneWeinberger}, Li and Yau
\cite{LiYauEigenvalues} and Kannan--Lov\'asz--Simonovits \cite{KLS}.
In fact, our estimates generalize to arbitrary Riemannian manifolds
satisfying our convexity assumptions, whereas the localization
method is confined to Euclidean space (and a few other special
manifolds). Using our
stability result, we are able to give a geometric proof of a recent
lower bound on $D_{Che}$ due to S. Bobkov \cite{BobkovVarianceBound}.
We also note that a recent result of Sasha Sodin
\cite{SodinLpIsoperimetry}, implying that $D_{Che}$ is uniformly bounded for
the suitably scaled unit-balls of $\ell_p^n$ for $p \in [1,2]$, is now an immediate
consequence of our Main Theorem together with a result of Schechtman and Zinn \cite{SchechtmanZinn2}.

\subsection{Ingredients in Proof of Main Theorem}

All of the four statements in our Main Theorem \ref{thm:Main} can be equivalently (up to universal constants)
rewritten using a single unified framework in terms of $(p,q)$ Poincar\'e inequalities:
\begin{dfn*}
The space $(\Omega,d,\mu)$ is said to satisfy a $(p,q)$ Poincar\'e inequality if:
\[
\exists D>0 \; \text{ such that } \; \forall f \in \F \;\;\;\;  D
\norm{f - M_\mu f}_{L_p(\mu)} \leq \norm{\abs{\nabla f}}_{L_q(\mu)} ~.
\]
The best possible constant $D$ above is denoted by $D_{p,q} = D_{p,q}(\Omega,d,\mu)$.
\end{dfn*}
We prefer to use the median $M_\mu$ in our definition for reasons
which will become apparent in Section \ref{sec:p-q}. It is known and
easy to establish that $D_{Poin} \simeq
D_{2,2}$, $D_{Che} = D_{1,1}$, $D_{FM} \simeq D_{1,\infty}$, so our
Main Theorem can be restated as the claim that all $(p,q)$
Poincar\'e inequalities in the range $1\leq p \leq q \leq \infty$
are equivalent under our convexity assumptions (see Theorem
\ref{thm:main-pq}).

\medskip

The convexity assumptions are used in an essential way in the proof of the Main Theorem in several
separate places. First, we employ the $CD(0,\infty)$ condition via the semi-group gradient estimates used by
Ledoux in his proof of Theorem \ref{thm:Ledoux}. Contrary to previous approaches, which could only
deduce isoperimetric information from functional inequalities with a $\norm{\abs{\nabla
f}}_{L_q(\mu)}$ term with $q=2$ (see \cite[p. 3]{BartheKolesnikov} and the references therein), we
can handle arbitrary $q\geq 1$ (and although we do not pursue this direction here, more general Orlicz norms too).
To demonstrate that our estimates are sharp, we remark that the isoperimetric inequalities we obtain are in fact equivalent (up to universal constants) to the $(p,q)$ Poincar\'e inequalities used to derive them. This is summarized in Theorem \ref{thm:1st-ingr}, which generalizes Theorems \ref{thm:Cheeger}, \ref{thm:GM}, \ref{thm:Buser} and \ref{thm:Ledoux} above into a single unified framework. Using this, we deduce from the First-Moment inequality ($p=1, q = \infty$ above) that:
\begin{equation} \label{eq:Intro-t2}
 \tilde{I}(t) \geq c D_{FM} t^2 \;\;\; \forall t \in [0,1/2] ~.
\end{equation}
To deduce Cheeger's isoperimetric inequality from (\ref{eq:Intro-t2}), we need to use our convexity
assumptions for the second time. We employ the following series of results in Riemannian Geometry,
due to numerous groups of authors
\cite{BavardPansu,GallotIsoperimetricInqs,MorganJohnson,SternbergZumbrun,Kuwert,BayleRosales,BayleThesis, MorganManifoldsWithDensity, BobkovExtremalHalfSpaces}, who proved them under
increasingly general conditions.
A detailed survey of these results
may be found in the Appendix. We learned about these results from the PhD Thesis of V.
Bayle \cite{BayleThesis}, which was referenced to us by Sasha Sodin, to whom we are indebted. In the formulation below, we use a slightly more general notion of smooth convexity assumptions, which is defined in Section \ref{sec:AA}.

\begin{thm}[Bavard--Pansu, B\'erard--Besson--Gallot, Gallot, Morgan--Johnson,
Sternberg--Zumbrun, Kuwert, Bayle--Rosales, Bayle, Morgan, Bobkov]
\label{thm:Intro-concavity} Under our generalized smooth convexity assumptions, the
isoperimetric profile $I = I_{(\Omega,d,\mu)}$ is concave on
$(0,1)$. Moreover, when $\mu$ is in addition uniform on $\Omega \subset (M,g)$, then $I^{n/(n-1)}$ is concave on $[0,1]$, where $n$ is the dimension of $M$.
\end{thm}

It is not hard to show (see Section \ref{sec:AA}) that the isoperimetric profile $I$ is continuous under
very general assumptions. It then follows by a general argument (e.g. Corollary
\ref{cor:profile-symmetry}) that $I$ must be symmetric about the point $1/2$. Hence, the concavity
of $I$ implies that $D_{Che} = 2 I(1/2)$ under our convexity assumptions. It is then immediate to
deduce Cheeger's isoperimetric inequality from (\ref{eq:Intro-t2}). In fact, a stronger statement can be deduced
when $\mu$ is uniform on $\Omega$ (see Remark \ref{rem:uniform}).

\medskip

A final ingredient in the proof is an approximation argument to handle non-smooth densities, which are typical
in applications as well as essential for handling uniform measures on bounded domains (with possibly non-smooth boundaries). Contrary to many results in Convexity Theory, where approximation arguments are standard, easy and usually omitted, the isoperimetric profile and the Cheeger constant are delicate objects, which in general are \emph{not stable} under approximation in the natural
total-variation metric (see Section \ref{sec:AA}). We therefore employ our convexity assumptions one last time, and provide in Section \ref{sec:AA} a careful argument for deducing the Main Theorem \ref{thm:Main} without any smoothness assumptions, and a different approximation procedure for extending Theorem \ref{thm:Intro-concavity}, which in particular applies to the entire class of log-concave measures in Euclidean space.

\medskip

The rest of this work is organized as follows. In Section \ref{sec:p-q}, we reformulate the Main Theorem in terms
of an equivalence between $(p,q)$ Poincar\'e inequalities, and using Theorem \ref{thm:Intro-concavity}, reduce it
to the statement of Theorem \ref{thm:1st-ingr}. The semi-group argument for proving Theorem \ref{thm:1st-ingr} is
described in Section \ref{sec:semi-group}. Further interpretations and an extension of the Main Theorem are described
in Section \ref{sec:extensions}. Applications for the spectral gap under our convexity assumptions are described
in Section \ref{sec:cor}. We conclude with an approximation argument for disposing of our smoothness assumptions in
Section \ref{sec:AA}, and an Appendix describing in more detail the results summarized in the statement of Theorem \ref{thm:Intro-concavity}.

\bigskip

\noindent \textbf{Acknowledgements.} I would like to thank Professor
Gideon Schechtman and the Weizmann Institute of Science where this
research project commenced during the last months of my PhD studies.
I would also like to thank Professor Jean Bourgain and the Institute
for Advanced Study for providing the perfect research environment.
Most especially, I would like to thank
Sasha Sodin for his invaluable help - acquainting me with
capacities, suggesting to look at the PhD Thesis of Bayle and
Ledoux's semi-group argument, countless other references, many
informative conversations and comments on this manuscript. I am also
grateful to Franck Barthe for his kind hospitality, remarks and advice, Bo'az Klartag for several references,
discussions and insightful remarks, Professors Sergey Bobkov and Michel Ledoux for their comments and advice, and
Professor David Jerison for several interesting conversations and suggestions. My gratitude also extends to the anonymous referees for their extraordinary careful reading and helpful suggestions, which greatly improved the presentation of this work.
Finally, I would also like to thank Professors David Jerison, Erwin Lutwak, Assaf Naor,
Vladimir Pestov and Santosh Vempala for their invitations to give talks on this work in its early development.

\section{$(p,q)$ Poincar\'e Inequalities} \label{sec:p-q}

We start by rewriting some of the statements of the Main Theorem \ref{thm:Main}.

We will use the following notation. A function $N: \Real_+ \rightarrow \Real_+$ will
be called a Young function if $N(0)=0$ and $N$ is convex increasing. Besides the classical Young functions $t^p$ ($p\geq 1$),
we will also frequently use the function $\Psi_1(t) = \exp(t)-1$. Given a Young function $N$,
the Orlicz norm $N(\mu)$ associated to $N$ is defined as:
\[
 \norm{f}_{N(\mu)} := \inf \set{ v>0 ; \int_\Omega N(|f|/v) d\mu \leq 1} ~.
\]

\begin{lem} \label{lem:E-M}
Let $N(\mu)$ denote an Orlicz norm associated to the Young function $N$. Then:
\[
 \frac{1}{2} \norm{f - E_\mu f}_{N(\mu)} \leq \norm{f - M_\mu f}_{N(\mu)} \leq 3 \norm{f - E_\mu
f}_{N(\mu)} ~.
\]
\end{lem}
\begin{proof}
Note that $\norm{1}_{N(\mu)} = 1 / N^{-1}(1)$. First, by Jensen's
inequality (applied twice):
\[
 \abs{E_\mu f - M_\mu f} \leq E_\mu(\abs{f - M_\mu f}) \leq N^{-1}(1) \norm{f - M_\mu f}_{N(\mu)} ~,
\]
hence:
\[
 \norm{f - E_\mu f}_{N(\mu)} \leq \norm{f - M_\mu f}_{N(\mu)} + \frac{\abs{E_\mu f - M_\mu
f}}{N^{-1}(1)} \leq 2  \norm{f - M_\mu f}_{N(\mu)} ~.
\]
Next, we may assume that $M_\mu f \geq E_\mu f$ (otherwise exchange
$f$ by $-f$). By the Markov-Chebyshev inequality:
\[
\frac{1}{2} \leq \mu( f \geq M_\mu f ) \leq \mu( \abs{f - E_\mu f} \geq M_\mu f -
E_\mu f ) \leq 1 \; / \; N\brac{\frac{M_\mu f -
E_\mu f}{\norm{f - E_\mu f}_{N(\mu)}}} ~,
\]
hence:
\[
 \abs{M_\mu f - E_\mu f} \leq N^{-1}(2) \norm{f - E_\mu f}_{N(\mu)} ~,
\]
and we deduce that:
\[
  \norm{f - M_\mu f}_{N(\mu)} \leq \norm{f - E_\mu f}_{N(\mu)} + \frac{\abs{E_\mu f - M_\mu
f}}{N^{-1}(1)} \leq \brac{1+\frac{N^{-1}(2)}{N^{-1}(1)}}  \norm{f - E_\mu f}_{N(\mu)} ~.
\]
We conclude by noting that $\frac{N^{-1}(2)}{N^{-1}(1)} \leq 2$ since $N$ is convex.
\end{proof}

The last lemma implies that we can pass back and forth between using
the median $M_\mu$ and the expectation $E_\mu$ when excluding
constant functions in our functional inequalities, at the expense of
losing a universal constant. We therefore see that Poincar\'e's
inequality is equivalent (up to constants) to the inequality:
\begin{equation} \label{eq:pq-Poin}
\forall f \in \F \;\;\; D^M_{Poin} \norm{f - M_\mu f}_{L_2(\mu)} \leq \norm{\abs{\nabla
f}}_{L_2(\mu)} ~,
\end{equation}
(and in fact in this case one clearly has $D_{Poin} \geq
D^M_{Poin}$). The next lemma, due to Maz'ya \cite{MazyaBook} and
Federer and Fleming \cite{FedererFleming} (see also \cite{BobkovHoudre}
for a careful derivation), rewrites Cheeger's isoperimetric
inequality in functional form:

\begin{lem}[Maz'ya, Federer--Fleming, Bobkov--Houdr\'e] \label{lem:Cheeger=L1}
Cheeger's isoperimetric inequality (with $D_{Che}$) holds iff:
\begin{equation} \label{eq:pq-Che}
\forall f \in \F \;\;\; D_{Che} \norm{f - M_\mu f}_{L_1(\mu)} \leq \norm{\abs{\nabla
f}}_{L_1(\mu)} ~.
\end{equation}
\end{lem}

\begin{proof}[Sketch of Proof following Bobkov--Houdr\'e \cite{BobkovHoudre}]
It is easy to show that Cheeger's isoperimetric inequality is recovered by applying
(\ref{eq:pq-Che}) to Lipschitz functions which approximate $\chi_A$, the
characteristic function of a Borel set $A$, in an appropriate sense. Conversely, the co-area
formula, which for general metric probability spaces becomes an inequality (see
\cite{BobkovHoudre}), implies for $f \in \F$ with $M_\mu f = 0$:
\begin{eqnarray*}
& & \int \abs{\nabla f} d\mu \geq \int_{-\infty}^\infty \mu^+\set{f > t} dt \\
&\geq& D_{Che} \brac{\int_{-\infty}^0 (1-\mu\set{f > t}) dt + \int_0^\infty \mu\set{f > t} dt }
= D_{Che} \int \abs{f} d\mu ~.
\end{eqnarray*}
\end{proof}

Since for a 1-Lipschitz function $f$, $\norm{\abs{\nabla f}}_{L_\infty(\mu)} \leq 1$,
our First-Moment inequality is clearly equivalent to:
\begin{equation} \label{eq:pq-FM}
 \forall f \in \F \;\;\; D^M_{FM} \norm{f - M_\mu f}_{L_1(\mu)} \leq \norm{\abs{\nabla
f}}_{L_\infty(\mu)} ~,
\end{equation}
in the sense that $D_{FM} \simeq D^M_{FM}$ where $D^M_{FM}$ is the best constant above.

\begin{rem} \label{rem:grad-defn}
The above functional reformulations remain valid for general metric probability spaces $(\Omega,d,\mu)$,
in which case we interpret $\abs{\nabla f}$ for any $f \in \F$ as the following Borel function:
\[
 \abs{\nabla f}(x) := \limsup_{d(y,x) \rightarrow 0+} \frac{|f(y) - f(x)|}{d(x,y)} ~.
\]
(and we define it as 0 if $x$ is an isolated point - see \cite[pp. 184,189]{BobkovHoudre}
for more details).
\end{rem}

\medskip

With the above reformulations (\ref{eq:pq-Poin}), (\ref{eq:pq-Che}),
(\ref{eq:pq-FM}) serving as motivation, the reasons behind our
definition of $(p,q)$ Poincar\'e inequalities in the Introduction
are now clear. Note that $D_{Che} = D_{1,1}$, $D^M_{Poin} = D_{2,2}$ and $D^M_{FM} = D_{1,\infty}$.
We can now restate our Main Theorem \ref{thm:Main} as follows:
\begin{thm} \label{thm:main-pq}
Under our convexity assumptions, all $(p,q)$ Poincar\'e inequalities are equivalent in the
range $1 \leq p \leq q \leq \infty$. More precisely, for any other $1 \leq p' \leq q' \leq \infty$:
\[
D_{p,q} \leq C p' D_{p',q'} ~,
\]
where $C>0$ is a universal constant.
\end{thm}

In fact, a more precise dependence on $p$ and $p'$ may be obtained in
some cases. For instance, clearly $D_{p',q'} \geq D_{p,q}$ if $p' \leq p$ and $q' \geq q$ without any further convexity assumptions
(by Jensen's inequality), so we see that the First-Moment inequality ($(1,\infty)$ case) is the weakest among all $(p,q)$ Poincar\'e
inequalities in the above range. Another immediate observation is given by:

\begin{prop} \label{prop:pq-increase}
Let $0< p \leq p' \leq \infty$ and $0 < q \leq q' \leq \infty$ be such that:
\[
 \frac{1}{p} - \frac{1}{q} = \frac{1}{p'} - \frac{1}{q'} ~.
\]
Then \emph{without} any further convexity assumptions, $D_{p',q'} \geq \frac{p}{p'} D_{p,q}$.
\end{prop}
\begin{proof}
Let $g \in \F$ denote a function with $M_\mu g = 0$. Define $f = \text{sign}(g) |g|^{p'/p}$, and apply the
$(p,q)$ Poincar\'e inequality to $f$. Clearly $M_\mu f = 0$, so we obtain by H\"{o}lder's
inequality:
\[
D_{p,q} \norm{g}^{p'/p}_{L_{p'}(\mu)} \leq \frac{p'}{p} \norm{|g|^{p'/p - 1} |\nabla g|}_{L_q(\mu)}
\leq \frac{p'}{p} \norm{g}^{p'/p - 1}_{L_{p'}(\mu)} \norm{|\nabla g|}_{L_{q'}(\mu)} ~,
\]
from which the assertion follows.
\end{proof}

\begin{cor}
Maz'ya--Cheeger inequality: $D_{Poin} \geq D_{Che} / 2$.
\end{cor}
\begin{proof}
 \[
  D_{Poin} \geq D^M_{Poin} = D_{2,2} \geq D_{1,1}/2 = D_{Che}/2.
 \]
\end{proof}

\begin{cor} \label{cor:prove-GM}
Gromov--Milman inequality: $D_{Exp} \geq c D_{Poin}$.
\end{cor}
\begin{proof}
Since $D_{Poin} \simeq D_{2,2}$, we conclude by Proposition \ref{prop:pq-increase} that $D_{p,p}
\geq c D_{Poin}/p$ for every $2 \leq p \leq \infty$.
Let $f$ be a 1-Lipschitz function. It is elementary to show (e.g. \cite{JSZ}) that
$1 / D_{Exp}(f)$ is equivalent (to within universal constants) to $\norm{f - E_\mu
f}_{\Psi_1(\mu)}$, and that $\norm{g}_{\Psi_1(\mu)}$ is in turn
equivalent to $\sup_{p \geq 1} \norm{g}_{L_p(\mu)}/{p}$. Employing Lemma \ref{lem:E-M} and using
the $(p,p)$ Poincar\'e inequalities:
\begin{eqnarray}
\nonumber & &  \frac{1}{D_{Exp}(f)} \simeq \norm{f - E_\mu f}_{\Psi_1(\mu)} \simeq \norm{f - M_\mu
f}_{\Psi_1(\mu)} \simeq \sup_{p \geq 1} \frac{\norm{f- M_\mu f}_{L_p(\mu)}}{p} \\
\nonumber & & \leq \sup_{p \geq
1} \frac{\norm{|\nabla f|}_{L_p(\mu)}}{\min(D_{2,2},p D_{p,p})} \leq \frac{C}{D_{Poin}} \sup_{p \geq
1} \norm{|\nabla f|}_{L_p(\mu)} = \frac{C}{D_{Poin}} ~,
\end{eqnarray}
since $f$ was assumed 1-Lipschitz. Taking supremum on all such functions $f$, we obtain the
conclusion.
\end{proof}
\begin{rem}
The exact same proof shows that $D_{Exp} \geq c_r D_{r,r}$, for arbitrary $r \geq 1$.
\end{rem}

We have seen that passing from $(p,q)$ to $(p',q')$ is manageable if
$q' \geq q$ (perhaps under some additional assumptions on $p,p'$)
without any convexity assumptions. Unfortunately, we are interested
in the case $q' < q$, for which an analogous statement to
Proposition \ref{prop:pq-increase} is simply false without any
additional assumptions (counter examples are easy to construct, as
in the Introduction). Our first ingredient in the proof of Theorem \ref{thm:main-pq}
states that our convexity assumptions already suffice to extend
Proposition \ref{prop:pq-increase} to the case $q' < q$, $p' < p$:

\begin{thm} \label{thm:1st-ingr}
Let $0 < p \leq \infty$, $1 \leq q \leq \infty$, and set $r = 1 +
\frac{1}{p} - \frac{1}{q}$. Assume that $\frac{1}{2} \leq r \leq 2$.
Then under our smooth convexity assumptions, the following
statements are equivalent:
\begin{enumerate}
 \item
\[
\forall f \in \F \;\;\;\;  D_{p,q} \norm{f - M_\mu f}_{L_p(\mu)}
\leq \norm{\abs{\nabla f}}_{L_q(\mu)} ~,
\]
\item
\[
 \tilde{I}(t) \geq D'_r t^r \;\;\; \forall t \in [0,1/2] ~,
\]
\end{enumerate}
where the best constants $D_{p,q}$ and $D'_r$ above satisfy:
\begin{equation} \label{eq:linear-decay}
 c_1 D_{p,q} \leq D'_r \leq c_2 p D_{p,q},
\end{equation}
for some universal constants $c_1,c_2>0$. \\
In fact, the direction
$(2)\Rightarrow(1)$ holds for $p \geq q$ without any convexity
assumptions.
\end{thm}

Note that when $p=q=2$, the direction $(2) \Rightarrow (1)$ reduces
(up to constants) to Theorem \ref{thm:Cheeger} (Maz'ya--Cheeger
inequality), and the direction $(1) \Rightarrow (2)$ to the
Buser--Ledoux Theorems \ref{thm:Buser},\ref{thm:Ledoux}. A
generalization of Theorem \ref{thm:1st-ingr} involving general
Orlicz norms will be derived in
\cite{EMilmanRoleOfConvexityInFunctionalInqs}.

\smallskip

There is essentially no novel content in the direction $(2)
\Rightarrow (1)$, which follows from the methods of Maz'ya
\cite[p. 89]{MazyaBook} and Federer--Fleming \cite{FedererFleming} (see
also \cite{BobkovHoudreMemoirs}). These authors deduced the optimal
constant in the Gagliardo inequality ($q=1$, $p=\frac{n}{n-1}$), as well as
the Sobolev inequalities ($1<q<n$, $p= \frac{qn}{n-q}$), from the isoperimetric inequality
in $\Real^n$ ($r=\frac{n-1}{n}$), using the following clever
generalization of Lemma \ref{lem:Cheeger=L1}:

\begin{prop}[Maz'ya, Federer--Fleming, Bobkov--Houdr\'e] \label{prop:isop-L1-inqs}
Let $0<r\leq 1$. Without any convexity
assumptions, the $(1/r,1)$ Poincar\'e inequality:
\[
\forall f \in \F \;\;\;\;  D \norm{f - M_\mu f}_{L_{1/r}(\mu)} \leq
\norm{\abs{\nabla f}}_{L_1(\mu)}
\]
is equivalent to the following isoperimetric inequality:
\[
 \tilde{I}(t) \geq D t^r \;\;\; \forall t \in [0,1/2] ~.
\]
\end{prop}

Combining Propositions \ref{prop:isop-L1-inqs} and
\ref{prop:pq-increase}, the direction $(2) \Rightarrow (1)$ for $p
\geq q$ (equivalently $r \leq 1$) immediately follows without any
further assumptions. For the case $p < q$, it is almost possible to
avoid using the convexity assumptions, but not completely. Instead,
we employ Theorem \ref{thm:Intro-concavity} on the concavity of $I$
under our (smooth) convexity assumptions, and deduce from $(2)$ that
in fact $\tilde{I}(t) \geq \frac{D_r'}{2^{r-1}} t$. The latter is
equivalent by Lemma \ref{lem:Cheeger=L1} to the statement $D_{1,1}
\geq \frac{D_r'}{2^{r-1}}$, and by using Proposition
\ref{prop:pq-increase} and Jensen's inequality, we deduce:
\[
D_{p,q} \geq D_{p,p} \geq \frac{D_{1,1}}{p} \geq \frac{D_r'}{2^{r-1}
p} \geq \frac{D_r'}{2 p} ~.\] The proof of $(2) \Rightarrow (1)$ is
thus complete.

\smallskip

Before proceeding to the proof of the direction $(1)\Rightarrow(2)$
(this will be the focus of the next section), let us recall how
Theorem \ref{thm:1st-ingr} coupled with Theorem
\ref{thm:Intro-concavity} conclude the proof of Theorem
\ref{thm:main-pq} and hence of our Main Theorem \ref{thm:Main}:

\begin{proof}[Proof of Theorem \ref{thm:main-pq}]
By an approximation argument we develop in Section \ref{sec:AA}, it is enough to prove
the theorem under our \emph{smooth} convexity assumptions.

By Jensen's inequality, $D_{1,\infty} \geq D_{p,q}$ in the range $1
\leq p \leq q \leq \infty$. Employing our (smooth) convexity
assumptions, the direction $(1) \Rightarrow (2)$ of Theorem
\ref{thm:1st-ingr} implies:
\begin{equation} \label{eq:quadratic}
\tilde{I}(t) \geq c D_{1,\infty} t^2 \;\;\; \forall t \in [0,1/2] ~.
\end{equation}
Using our (smooth) convexity assumptions for the second time, Theorem
\ref{thm:Intro-concavity} asserts that $I$ is concave on $(0,1)$.
Since $I$ is also symmetric about $1/2$ (see Corollary
\ref{cor:profile-symmetry}), we immediately deduce that:
\[
\tilde{I}(t) \geq \frac{c}{2} D_{1,\infty} t \;\;\; \forall t \in [0,1/2] ~,
\]
which is exactly Cheeger's isoperimetric inequality, and is
identical to stating $D_{1,1} \geq \frac{c}{2} D_{1,\infty}$. Using
Proposition \ref{prop:pq-increase} and Jensen's inequality if
necessary, we can pass from this to an arbitrary $(p',q')$
inequality in the range $1 \leq p' \leq q' \leq \infty$.
\end{proof}

\begin{rem} \label{rem:uniform}
Note that when $\mu$ is the uniform measure on $\Omega$, Theorem \ref{thm:Intro-concavity} in fact
ensures that $I^{\frac{n}{n-1}}$ is concave, so we may deduce from (\ref{eq:quadratic}) that in
fact:
\[
 \tilde{I}(t) \geq  \frac{c}{2^{\frac{n+1}{n}}} D_{1,\infty} t^{\frac{n-1}{n}} \;\;\; \forall t \in
[0,1/2] ~.
\]
Proposition \ref{prop:isop-L1-inqs} implies that the latter
isoperimetric inequality is equivalent to a $(\frac{n}{n-1}, 1)$
Poincar\'e inequality. Hence, it is clear that in this case, both
our Main Theorem \ref{thm:Main} and Theorem \ref{thm:main-pq} can be
strengthened.
\end{rem}

\section{The Semi-Group Argument} \label{sec:semi-group}

In this section, we prove the direction $(1) \Rightarrow (2)$ of
Theorem \ref{thm:1st-ingr}. Our proof closely follows Ledoux's proof
\cite{LedouxSpectralGapAndGeometry} of Theorem \ref{thm:Ledoux}.

Given a smooth complete oriented connected Riemannian manifold
$\Omega = (M,g)$ equipped with a probability measure $\mu$ with
density $d\mu = \exp(-\psi) dvol_M$, $\psi \in C^2(M)$, we define
the associated Laplacian $\Delta_{(\Omega,\mu)}$ by:
\begin{equation} \label{eq:Laplacian-def}
 \Delta_{(\Omega,\mu)} := \Delta_{\Omega} - \nabla \psi \cdot \nabla,
\end{equation}
where $\Delta_{\Omega}$ is the usual Laplace-Beltrami operator on
$\Omega$. $\Delta_{(\Omega,\mu)}$ acts on $\B(\Omega)$, the space of
bounded smooth real-valued functions on $\Omega$.
Let $(P_t)_{t \geq 0}$ denote the semi-group associated to the
diffusion process with infinitesimal generator
$\Delta_{(\Omega,\mu)}$ (cf.
\cite{DaviesSemiGroupBook,LedouxLectureNotesOnDiffusion}),
characterized by the following system of second order differential
equations:
\[
 \frac{d}{dt} P_t(f) = \Delta_{(\Omega,\mu)} (P_t(f)) \;\;\;\; P_0(f) = f \;\;\; \forall f \in
\B(\Omega)~.
\]
For each $t \geq 0$, $P_t : \B(\Omega) \rightarrow \B(\Omega)$ is a
bounded linear operator and its action naturally extends to the
entire $L_p(\mu)$ spaces ($p\geq 1$). We collect several elementary
properties of these operators:
\begin{itemize}
\item
$P_t (1) = 1$.
 \item
$f \geq 0 \Rightarrow P_t (f) \geq 0$.
\item
$\int P_t (f) d\mu = \int f d\mu$.
\item
$\abs{P_t(f)}^p \leq P_t(\abs{f}^p)$ for all $p \geq 1$.
\end{itemize}

The following crucial dimension-free reverse Poincar\'e inequality
was shown by Bakry and Ledoux in \cite[Lemma 4.2]{BakryLedoux},
extending Ledoux's approach \cite{LedouxBusersTheorem} for proving
Buser's Theorem (see also \cite[Lemma 2.4]{BakryLedoux}, \cite[Lemma
5.1]{LedouxSpectralGapAndGeometry}). It may also be interpreted as a
weak, dimension-free, form of the Li--Yau parabolic gradient
inequality \cite{LiYauParabolicInq}.

\begin{lem}[Bakry--Ledoux] \label{lem:Ledoux}
Assume that the following Bakry-\'{E}mery Curvature-Dimension
condition holds on $\Omega$:
\begin{equation} \label{eq:Ledoux-condition}
Ric_g + Hess_g \psi \geq -K g ~, K \geq 0 ~.
\end{equation}
Then for any $t\geq 0$ and $f \in \B(\Omega)$, we have:
\[
 c(t) \abs{\nabla P_t (f)}^2 \leq P_t(f^2) - (P_t (f))^2
\]
pointwise, where:
\[
 c(t) = \frac{1 - \exp(-2Kt)}{K} \;\;\; (= 2t \;\; \text{ if } K=0).
\]
\end{lem}

In fact, the proof of this lemma is very general and extends to the
abstract framework of diffusion generators, as developed by Bakry
and \'{E}mery \cite{BakryEmery}. We comment that in the Riemannian
setting, it is known \cite{QianGradientEstimateWithBoundary} (see
also
\cite{HsuGradientEstimateWithBoundary,WangGradientEstimateWithBoundary})
that the gradient estimate of Lemma \ref{lem:Ledoux} is preserved
when restricting to a locally convex domain (as defined in the
Appendix) with smooth boundary; we refer to Sturm \cite[Proposition
4.15]{SturmCD1} for a general statement about closedness of the
Bakry-\'{E}mery Curvature-Dimension condition in an arbitrary metric
probability space.
The above lemma therefore holds under more general conditions,
namely when $\mu$ is supported on a locally convex domain $\Omega
\subset (M,g)$ with $C^2$ boundary, and $d\mu|_\Omega = \exp(-\psi)
dvol_M|_\Omega$, $\psi \in C^2(\overline{\Omega})$. In this case,
$\Delta_{\Omega}$ in (\ref{eq:Laplacian-def}) denotes the Neumann
Laplacian on $\Omega$, $\B(\Omega)$ denotes the space of
bounded smooth real-valued functions on $\Omega$
satisfying Neumann's boundary condition on $\partial \Omega$, and
Lemma \ref{lem:Ledoux} remains valid.

Our convexity assumptions are that $K=0$ in Lemma \ref{lem:Ledoux},
and this is what we will henceforth assume. It is clear that our
results in this section
may be extended to the case of $K > 0$, but we do not pursue this
direction in this work.

From Lemma \ref{lem:Ledoux}, it is immediate that for any $2 \leq q
\leq \infty$:
\begin{equation} \label{eq:L_q-bound}
 \norm{\abs{\nabla P_t (f)} }_{L_q(\mu)} \leq \frac{1}{\sqrt{2t}} \norm{f}_{L_q(\mu)},
\end{equation}
and using $q = \infty$, Ledoux easily deduces the following dual
statement \cite[(5.5)]{LedouxSpectralGapAndGeometry}:
\begin{cor}[Ledoux] \label{cor:Ledoux}
\begin{equation}
 \norm{ f - P_t (f)}_{L_1(\mu)} \leq \sqrt{2 t} \norm{ \abs{\nabla f} }_{L_1(\mu)}.
\end{equation}
\end{cor}

\begin{proof}[Proof of $(1) \Rightarrow (2)$ of Theorem \ref{thm:1st-ingr}]

First, our assumption on the range of $r$ implies that by applying
Proposition \ref{prop:pq-increase} if necessary, we may assume that
$p \geq 1, q \geq 2$ at the expense of an additional universal
constant appearing in (\ref{eq:linear-decay}). An additional
universal constant will appear on account of Lemma \ref{lem:E-M},
with which we pass to $E_\mu$ instead of $M_\mu$ in (1), so our
assumption now reads:
\begin{equation} \label{eq:sg-proof-assumption}
p \geq 1, q \geq 2 \;\; , \;\; \forall f \in \F \;\;\;\; D_{p,q}
\norm{f - E_\mu f}_{L_p(\mu)} \leq \norm{\abs{\nabla f}}_{L_q(\mu)}.
\end{equation}

Let $A$ denote an arbitrary Borel set in $\Omega$, and let
$\chi_{A,\eps}(x) := (1 - \frac{1}{\eps} d_g(x,A)) \vee 0$ denote a
continuous approximation in $\Omega$ to the characteristic function
$\chi_A$ of $A$. Clearly:
\[
 \frac{\mu(A_\eps) - \mu(A)}{\eps} \geq \int \abs{\nabla{\chi_{A,\eps}}} d\mu.
\]
Applying Corollary \ref{cor:Ledoux} to functions in $\B(\Omega)$
which approximate $\chi_{A,\eps}$ (in say $W^{1,1}(\Omega,\mu)$) and
passing to the limit inferior as $\eps \rightarrow 0$, it follows
that:
\[
 \sqrt{2t} \mu^+(A) \geq \int \abs{\chi_A - P_t(\chi_A) } d\mu.
\]

We start by rewriting the right hand side above as:
\begin{multline*}
\int_A (1 - P_t(\chi_A)) d\mu + \int_{\Omega \setminus A}
P_t(\chi_A) d\mu = 2\brac{\mu(A) - \int_A
P_t(\chi_A) d\mu} \\
= 2 \brac{\mu(A)(1-\mu(A)) - \int_\Omega (P_t(\chi_A) -
\mu(A))(\chi_A - \mu(A)) d\mu} ~.
\end{multline*}
Note that by H\"{o}lder's inequality (recall that $p\geq 1$) and our
assumption (\ref{eq:sg-proof-assumption}):
\begin{eqnarray}
\nonumber \int_\Omega (P_t(\chi_A) - \mu(A))(\chi_A - \mu(A)) d\mu &
\leq & \norm{P_t(\chi_A) - \mu(A)}_{L_p(\mu)}
\norm{\chi_A - \mu(A)}_{L_{p^*}(\mu)} \\
\nonumber &\leq& D_{p,q}^{-1} \norm{\abs{\nabla
P_t(\chi_A)}}_{L_q(\mu)} \norm{\chi_A - \mu(A)}_{L_{p^*}(\mu)} ~.
\end{eqnarray}
Using (\ref{eq:L_q-bound}) (recall that $q \geq 2$) to estimate
$\norm{\abs{\nabla P_t (\chi_A) }}_{L_q(\mu)}$, we conclude that:
\begin{equation} \label{eq:sg-proof-conclude}
 \sqrt{2t} \mu^+(A) \geq 2 \brac{\mu(A)(1-\mu(A)) - \frac{1}{\sqrt{2t} D_{p,q}} \norm{\chi_A
- \mu(A)}_{L_{q}(\mu)} \norm{\chi_A - \mu(A)}_{L_{p^*}(\mu)}} ~.
\end{equation}
We may now optimize on $t$. Using the rough estimate:
\[
\norm{\chi_A - \mu(A)}_{L_{s}(\mu)} \leq 2
\brac{\mu(A)(1-\mu(A))}^{1/s}
\]
for $s \geq 1$, we evaluate (\ref{eq:sg-proof-conclude}) at time:
\[
 t = \frac{32}{D_{p,q}^2} \brac{\mu(A)(1-\mu(A))}^{2(1/q-1/p)}
\]
and deduce:
\[
 \mu^+(A) \geq \frac{D_{p,q}}{8} \brac{\mu(A)(1-\mu(A))}^{2-1/q-1/p^*} \geq \frac{D_{p,q}}{8
\cdot 2^{r}} \min(\mu(A),1-\mu(A))^r ~,
\]
where $r = 2 - 1/q - 1/p^* = 1 + 1/p - 1/q$. Since $r \leq 2$, this
concludes the proof.
\end{proof}

\begin{rem}
As evident from the proof, for deducing the direction $(1) \Rightarrow (2)$ of Theorem \ref{thm:1st-ingr}, the definition of smooth convexity assumptions given in the Introduction may be extended to encompass the more general case treated in this section. Moreover, it is possible to provide an approximation argument for deducing this direction without any smoothness assumptions. We provide the argument in \cite{EMilmanRoleOfConvexityInFunctionalInqs} and omit it here,
since it is not required for the results of this work.
\end{rem}

\section{Interpretations and Extensions} \label{sec:extensions}

In this section, we provide some further interpretations and
extensions of our Main Theorem, which will also be needed for the applications of the next section.
We assume throughout this section that our convexity assumptions on $(\Omega,d,\mu)$ are satisfied.

Lemma \ref{lem:Cheeger=L1} demonstrates that if $A$ is a set with
$\mu(A) \leq 1/2$ on which the minimal ratio $D_{Che} = \mu^+(A) /
\mu(A)$ in Cheeger's isoperimetric inequality is attained (or nearly attained),
then the function $f=\chi_A$ (or the sequence of
Lipschitz functions which approximate it) attains the same (nearly)
minimal ratio
\begin{equation} \label{eq:ratio}
\int \abs{\nabla f} d\mu \; / \int \abs{f} d\mu
\end{equation}
among all functions $f \in \F$ with $M_\mu f = 0$. Clearly $\chi_A$
(or its approximating sequence) is far from being $1$-Lipschitz. If
on the other hand we define:
\begin{equation} \label{eq:Lip-form}
f(x) = d(x,\Omega \setminus A),
\end{equation}
which \emph{is} a $1$-Lipschitz function, it is not clear that it
will have a small ratio in (\ref{eq:ratio}). Our Main Theorem \ref{thm:Main} (together with Lemma \ref{lem:E-M}) states
that under our convexity assumptions, any $1$-Lipschitz
function $f_0$ on $(\Omega,d)$ with $M_\mu f_0 = 0$ which is (essentially) optimal in the First-Moment
inequality (say $\int |f_0| d\mu \geq 1/(3 D^M_{FM})$), also essentially minimizes the ratio in (\ref{eq:ratio}). Moreover, using the co-area formula as in Lemma
\ref{lem:Cheeger=L1} and applying our Main Theorem, we have:
\[
D_{Che} \leq \frac{\int_{-\infty}^\infty \mu^+\set{f_0 > t}
dt}{\int_{-\infty}^\infty \min(\mu\set{f_0 > t},1-\mu\set{f_0 > t})
dt} \leq
\frac{\int \abs{\nabla f_0} d\mu}{\int \abs{f_0} d\mu} \leq 3 D^M_{FM} \leq C D_{Che} ~,
\]
from which we also see that the ratio
$\mu^+(A_t) / \min(\mu(A_t),1 - \mu(A_t))$ for the ``average" level set $A_t$ of $f_0$
is essentially $D_{Che}$, the smallest possible.

Theorem \ref{thm:Lip-form} from the Introduction states that $f_0$ as above may in fact be chosen to be of the
form (\ref{eq:Lip-form}).

\begin{proof}[Proof of Theorem \ref{thm:Lip-form}]
Given a Borel set $A \subset \Omega$ with $\mu(A)\geq 1/2$, we
denote $g_A(x) = d(x,A)$. Clearly $g_A$ is $1$-Lipschitz and $M_\mu
g_A = 0$, so one direction follows immediately by Lemma \ref{lem:Cheeger=L1}:
\[
 D_{Che}(\Omega,d,\mu) \leq \frac{\int \abs{\nabla g_A} d\mu}{\int \abs{g_A} d\mu} \leq
\frac{1/2}{\int d(x,A) d\mu} ~.
\]
For the other direction, we employ our Main Theorem (and
Lemma \ref{lem:E-M}):
\[
 D_{Che}(\Omega,d,\mu) \geq c D^M_{FM}(\Omega,d,\mu) = \inf \frac{c}{\int |f | d\mu} ~,
\]
where the infimum is over all $1$-Lipschitz functions $f$ on
$(\Omega,d)$ with $M_\mu f = 0$. Denoting $A_1 = \set{f \leq 0}, A_2 = \set{f
\geq 0}$, we have $\mu(A_i) \geq 1/2$, $i=1,2$. By continuity
of $f$, $f|_{\partial A_1} \equiv 0$, $f|_{\partial A_2}
\equiv 0$ (even though it is possible that $\partial A_1 \neq
\partial A_2$), and since it is $1$-Lipschitz:
\[
\int |f | d\mu \leq \int_{\Omega \setminus A_2}
d(x,\partial A_2) d\mu + \int_{\Omega \setminus A_1} d(x,\partial
A_1) d\mu =  \int d(x,A_2) d\mu + \int d(x,A_1) d\mu ~.
\]
This concludes the proof.
\end{proof}

The next proposition will prove to be very useful for the applications of the next section.
We start with some notations. Given a Borel function $f$ on a
Borel probability space $(\Omega,\mu)$ and $\delta \in [0,1]$, let us denote by
$Q_\delta(f) = Q_{\mu,\delta}(f)$ the \emph{$\delta$-quantile} of
$f$:
\[
Q_\delta(f) := \inf\set{q \in \Real ; \mu\set{f \leq q} \geq \delta}
~.
\]
Let us also recall an inequality due to Paley and Zygmund
\cite{PaleyZygmund} (see also \cite[Chapter 2]{KahaneBook}), which
in its simplest form reads as follows:
\begin{lem}\label{lem:PZ}
Let $f$ denote a Borel function on $\Omega$, and assume that:
\[
\exists D>0  \;\;\;\text{such that}\;\;\; \norm{f}_{L_2(\mu)} \leq D
\norm{f}_{L_1(\mu)} < \infty ~.
\]
Then for any $\theta \in (0,1)$, denoting $\eps(\theta) =
(1-\theta)^2/D^2$, one has $Q_{1-\eps(\theta)}(|f|) \geq \theta \norm{f}_{L_1(\mu)}$.
\end{lem}

\begin{prop} \label{prop:useful}
Let $f_0$ denote a $1$-Lipschitz function with either $M_\mu f_0 = 0$ and $\norm{f_0}_{L_1(\mu)} \geq 1/(2 D^M_{FM})$ or
$E_\mu f_0 = 0$ and $\norm{f_0}_{L_1(\mu)} \geq 1/(2 D_{FM})$. Then:
\begin{equation} \label{eq:Psi_1-condition}
\norm{f_0}_{\Psi_1(\mu)} \leq C_0 \norm{f_0}_{L_1(\mu)} ~,
\end{equation}
and consequently:
\begin{equation} \label{eq:L_1-Q}
Q_{1-\eps_0}(|f_0|) \geq \norm{f_0}_{L_1(\mu)} / 2 ~,
\end{equation}
for some universal constants $C_0 > 0$ and $0 < \eps_0 < 1$.
\end{prop}
\begin{proof}
Proceeding as in Corollary \ref{cor:prove-GM}, and using Lemma \ref{lem:E-M} and the Main Theorem:
\[
\norm{f_0}_{\Psi_1(\mu)} \simeq \norm{f_0 - E_\mu f_0}_{\Psi_1(\mu)} \simeq \frac{1}{D_{Exp(f_0)}} \leq
\frac{1}{D_{Exp}} \leq \frac{C}{\max(D_{FM},D^M_{FM})} \leq 2C
\norm{f_0}_{L_1(\mu)} ~.
\]
Consequently, it is easy to check that:
\[
\norm{f_0}_{L_2(\mu)} \leq 2 \norm{f_0}_{\Psi_1(\mu)} \leq D_0
\norm{f_0}_{L_1(\mu)} ~,
\]
for some universal constant $D_0>0$, and (\ref{eq:L_1-Q}) follows by Lemma \ref{lem:PZ} (with $\theta = 1/2$).
Note that our convexity assumptions necessarily imply that
$\norm{f_0}_{L_1(\mu)} < \infty$ (see Lemma \ref{lem:Borell}), so the appeal to Lemma \ref{lem:PZ} is
indeed legitimate.
\end{proof}

\begin{cor}
An arbitrarily slow uniform tail decay condition (\ref{eq:slow-tail-decay})
implies any of the statements of the Main Theorem \ref{thm:Main}, with $D_{Che},D_{Poin},D_{Exp},D_{FM}$
depending solely on $\alpha$. Moreover, $E_\mu f$ in (\ref{eq:slow-tail-decay}) may be replaced by $M_\mu f$.
\end{cor}
\begin{proof}
Given a $1$-Lipschitz function $f_0$ satisfying either of the assumptions of Proposition \ref{prop:useful},
these and (\ref{eq:L_1-Q}) imply that:
\[
\frac{1}{2 \max(D_{FM},D^M_{FM})} \leq \norm{f_0}_{L_1(\mu)} \leq 2
Q_{1-\eps_0}(|f_0|) ~.
\]
Consequently, the tail decay condition (\ref{eq:slow-tail-decay}) (whether stated with $E_\mu f$ or $M_\mu f$) ensures
that $\max(D_{FM},D^M_{FM}) \geq 1/(4 \alpha^{-1}(\eps_0)) > 0$, so by Lemma \ref{lem:E-M} the First-Moment concentration inequality is satisfied,
from which the other statements of the Main Theorem follow.
\end{proof}

\begin{rem} \label{rem:FM-Exp-Equivalent}
Using standard results in Convexity Theory (e.g. Borell's Lemma
\cite{Borell-logconcave}), it is well known that when $\mu$ is a
log-concave measure on $\Real^n$ and $f_0$ is a \emph{linear} (more
generally, convex homogeneous) functional, then (\ref{eq:Psi_1-condition}) is satisfied
with some universal constant $C>0$. In this sense, our essentially optimal $1$-Lipschitz function $f_0$
behaves like linear functionals. A conjecture of Kannan, Lov\'asz and Simonovits
which will be described in Section \ref{sec:cor}, states this even
more explicitly: linear functionals are essentially optimal in the
$(1,1)$ or $(2,2)$ Poincar\'e inequalities. Using our Main Theorem,
we now see that this conjecture is equivalent to
stating that linear functionals are essentially optimal in the
exponential concentration and First-Moment inequalities.
In this sense, the Main Theorem may be thought of as a qualitative step towards resolving the conjecture:
an essentially optimal function above has the form $f_0 = d(x,A)$ with $\mu(A) \geq 1/2$, and it remains to
show that one can choose $A$ to be a half-space (so that $f_0$ becomes linear).
\end{rem}

\section{Applications to Spectral Gap of Convex Domains} \label{sec:cor}

In this section, we provide several applications of our Main Theorem pertaining to the spectral gap $D_{Poin}^2(\Omega,d,\mu)$ of metric probability spaces satisfying our convexity assumptions. The results will be formulated in terms of the Cheeger constant $D_{Che}(\Omega,d,\mu)$, which by the Maz'ya--Cheeger inequality (Theorem \ref{thm:Cheeger}) and the Buser-Ledoux Theorems (\ref{thm:Buser} and \ref{thm:Ledoux}) is equivalent to $D_{Poin}(\Omega,d,\mu)$ under these assumptions (see also the approximation arguments of Section \ref{sec:AA} to handle non-smooth domains and densities). We will mostly restrict our attention to the case of $\Real^n$ with some fixed Euclidean structure $\abs{\cdot}$, although in some places we will mention our result in its full generality on Riemannian manifolds.

Given a bounded domain $\Omega \subset (M,g)$, we denote
the uniform probability measure on $\Omega$ by $\lambda_\Omega := \frac{vol_M|_\Omega}{vol_M(\Omega)}$. We will write $D_{Che}(\Omega)$, $D_{FM}(\Omega)$, and so on,  to denote
$D_{Che}(\Omega,\abs{\cdot},\lambda_\Omega)$,
$D_{FM}(\Omega,\abs{\cdot},\lambda_\Omega)$ for short. We will say
that $\Omega$ is a convex body if $\Omega$ is a convex bounded
domain in $(\Real^n,\abs{\cdot})$. We will sometimes not distinguish
between $\Omega$ and its closure $\overline{\Omega}$.

\subsection{Stability of $D_{Che}$ under Perturbations}

First, we would like to obtain a stability result for $D_{Che}(\Omega)$ (or equivalently $D_{Poin}(\Omega)$) for perturbations of $\Omega$. Clearly, without any further assumptions, there can be no such result (as seen by adding arbitrarily small ``necks'' to $\Omega$ as in the Introduction), so we restrict our attention to convex domains. In this case, our Main Theorem \ref{thm:Main} asserts that this is equivalent to obtaining a stability result for $D_{FM}(\Omega)$, which is much easier. To obtain the best quantitative bounds, we will also employ $D_{Exp}(\Omega)$.

\begin{lem} \label{lem:going-down}
Let $L \subset K \subset (\Real^n,\abs{\cdot})$, and assume that $L$
is a convex body. There exists a universal constant $c>0$ such that:
\[
\Vol{L} \geq v \Vol{K}  \;\;\; \Rightarrow \;\;\; D_{FM}(L) \geq \frac{c}{\log(1+1/v)}
D_{Exp}(K) ~.
\]
\end{lem}

\begin{proof}
Let $f_0$ denote a $1$-Lipschitz function on $L$ with $M_{\lambda_L} f_0 = 0$ so that
$\int |f_0| d\lambda_L \geq 1/(2 D^M_{FM}(L))$. Since $L$ is convex, we may clearly extend $f_0$ to a
$1$-Lipschitz function on $K$, say by defining $f_1 =
f_0(\text{Proj}_L x)$.
Here $\text{Proj}_L x$ denotes the unique (by convexity) $y$ in $L$ so that $d(x,L) = d(x,y)$. We may assume that
$E_{\lambda_K} f_1 \geq 0$ (otherwise exchange $f_0$ with $-f_0$). Note
that we can estimate $E_{\lambda_K} f_1$ as follows:
\begin{equation} \label{eq:estimate-on-M}
\frac{v}{2} \leq \lambda_K\set{f_1 \leq 0}
\leq \lambda_K\set{\abs{f_1 - E_{\lambda_K} f_1}\geq E_{\lambda_K}
f_1} \leq e \cdot \exp(- D_{Exp}(K) E_{\lambda_K} f_1 ) ~.
\end{equation}

By Proposition \ref{prop:useful}, there exists some universal $\eps_0>0$ so that $\norm{f_0}_{L_1(\lambda_L)} \leq Q_{\lambda_L,1-\eps_0}(|f_0|)$. Using this, the ratio between the volumes of $L$ and $K$, the triangle inequality,
the Markov-Chebyshev inequality and the estimate on $E_{\lambda_K} f_1$ in (\ref{eq:estimate-on-M}), we evaluate:
\begin{eqnarray*}
\frac{1}{2 D^M_{FM}(L)} & \leq & \norm{f_0}_{L_1(\lambda_L)} \leq Q_{\lambda_L,1-\eps_0}(|f_0|) \leq Q_{\lambda_K,1-\eps_0 v}(|f_1|) \leq  Q_{\lambda_K,1-\eps_0 v}(|f_1 - E_{\lambda_K} f_1|) + E_{\lambda_K} f_1 \\
&\leq& \log\brac{1+\frac{1}{\eps_0 v}} \norm{f_1- E_{\lambda_K} f_1}_{\Psi_1(\lambda_K)} + \frac{\log(2e/v)}{D_{Exp}(K)} \leq C_0 \frac{\log(1+1/v)}{D_{Exp}(K)} ~,
\end{eqnarray*}
where $C_0>0$ is some universal constant. Using Lemma \ref{lem:E-M} and (\ref{eq:pq-FM}), the assertion follows.
\end{proof}

\begin{lem} \label{lem:going-up}
Let $L \subset K \subset (\Real^n,\abs{\cdot})$, and assume that $L$
and $K$ are convex bodies. Then:
\[
\Vol{L} \geq v \Vol{K}  \;\;\; \Rightarrow \;\;\; D_{Che}(K) \geq
v^2 D_{Che}(L) ~.
\]
\end{lem}
\begin{proof}
Note that for any $1/2<p \leq 1$ and in fact even without assuming
that $L$ is convex:
\begin{equation} \label{eq:going-up-direct}
\Vol{L} \geq p \Vol{K} \;\; \Rightarrow \;\; D_{Che}(K) \geq (2p - 1)
D_{Che}(L) ~.
\end{equation}
Indeed, since $K$ is convex, by Theorem \ref{thm:Intro-concavity} (more precisely,
its extension to non-smooth domains or densities given by Theorem \ref{thm:main-approx} and Corollaries \ref{cor:approx-bodies},\ref{cor:approx-lc}) we
know that $D_{Che}(K) = 2 I_{(K,\abs{\cdot},\lambda_K)}(1/2)$. Given
a Borel set $A$ with $\lambda_K(A) = 1/2$, we have:
\[
\lambda_K^+(A) \geq p \lambda_L^+(A)
\geq p D_{Che}(L)  \min(\lambda_L(A),1-\lambda_L(A)) ~.
\]
By the assumption in (\ref{eq:going-up-direct}), $1-\frac{1}{2p}
\leq \lambda_L(A) \leq \frac{1}{2p}$, and from this we easily
deduce the conclusion in (\ref{eq:going-up-direct}). Iterating this
using a sequence of intermediate convex bodies (here we already need
to use that $L$ is convex) $L = L_0 \subset L_1 \subset \ldots
\subset L_m = K$ so that $\Vol{L_i} / \Vol{L_{i+1}} \geq v^{1/m} > 1/2$ (for example, assuming $0 \in L$, choose $L_i = (1+r_i)L \cap K$ for
appropriate $r_i \geq 0$), we obtain that:
\[
\Vol{L} \geq v \Vol{K} \;\; \Rightarrow \;\; D_{Che}(K) \geq (2 v^{1/m} -1)^m D_{Che}(L) ~.
\]
Taking the limit as $m \rightarrow \infty$ yields the claimed power of $2$ (even without any additional numerical
constant!).
\end{proof}

We can now immediately deduce Theorem \ref{thm:Intro-stability} from the Introduction.
Indeed, if $K,L$ denote two convex bodies in $(\Real^n,\abs{\cdot})$ such that:
\[
\Vol{K \cap L} \geq v_K \Vol{K} \;\;\;,\;\;\; \Vol{K \cap L} \geq v_L \Vol{L} ~,
\]
then applying Lemma \ref{lem:going-up}, the Main Theorem \ref{thm:Main} and Lemma \ref{lem:going-down}, we obtain:
\begin{eqnarray}
\nonumber D_{Che}(K) &\geq & v_K^2 D_{Che}(K \cap L) \geq c_1 v_K^2 D_{FM}(K \cap L)
\geq c_2 \frac{v_K^2}{\log(1+1/v_L)} D_{Exp}(L) \\
\label{eq:thm-inclusion-stability}
&\geq & c_3 \frac{v_K^2}{\log(1+1/v_L)} D_{Che}(L) ~,
\end{eqnarray}
for some universal constants $c_i > 0$, concluding the proof of Theorem \ref{thm:Intro-stability}. Of course a similar upper bound on $D_{Che}(K)$ is obtained by interchanging the roles of $K,L$.

\medskip

In Convexity Theory, many interesting ways are known to cut a convex
body $K$ so that its volume is preserved up to a constant (e.g. by
slabs, parallelepipeds, balls etc...). We see that all of these
preserve (up to a constant) $D_{Che}(K)$ (equivalently, the
spectral gap $D_{Poin}^2(K)$). A useful way to measure the distance
between two convex bodies is given by the following variant on the
usual geometric distance:
\begin{equation} \label{eq:geometric-distance}
 d_G(K,L) := \inf\set{a b \; ; \; \frac{1}{a} L \subset K \subset b L \;,\; a,b \geq 1} ~.
\end{equation}
Clearly in $(\Real^n,\abs{\cdot})$:
\[
 \frac{\Vol{L}}{\Vol{K}} \leq d_G(K,L)^n ~,
\]
so by passing from the outer to the inner body (in which case our estimates are logarithmic), we deduce:
\begin{cor} \label{cor:isomorphism}
Let $K,L$ denote two convex bodies in $(\Real^n,\abs{\cdot})$. If:
\[
 d_G(K,L) \leq 1 + \frac{s}{n}
\]
for some $1 \leq s \leq C_1 n$, where $C_1>0$ is some universal constant, then:
\[
 C_2 s D_{Che}(L) \geq  D_{Che}(K) \geq \frac{1}{C_2 s} D_{Che}(L) ~,
\]
where $C_2>0$ is another universal constant.
\end{cor}
\begin{proof}
Denoting $a,b$ the best constants in (\ref{eq:geometric-distance}) and applying Lemma \ref{lem:going-down}:
\[
 D_{Che}(K) \geq \frac{D_{Che}(bL)}{C \log(1+d_G(K,L)^n)} \geq \frac{D_{Che}(L)}{C' b s} ~,
\]
and since $b \leq d_G(K,L) \leq C_1 + 1$, the assertion follows.
\end{proof}

\medskip

Completely analogous results hold for absolutely
continuous log-concave probability measures $\mu$ on
$(\Real^n,\abs{\cdot})$. We will write $D_{Che}(\mu)$ (and so on) to denote $D_{Che}(\Real^n,\abs{\cdot},\mu)$ for short.
Lemmas \ref{lem:going-down} and \ref{lem:going-up}
were only formulated for uniform distributions $\lambda_K,\lambda_L$ on domains $K,L$, since in that case,
the condition:
\begin{equation} \label{eq:lem-cond}
L \subset K \text{ with } \Vol{L} \geq v \Vol{K}
\end{equation}
appearing in the assumptions of both lemmas has a clear and intuitive geometric meaning.
\begin{lem} \label{lem:gen-up-down}
Lemmas \ref{lem:going-down} and \ref{lem:going-up} remain valid for absolutely continuous
log-concave probability measures $\mu_K,\mu_L$ (replacing respectively $K,L$), if the
condition (\ref{eq:lem-cond}) in the assumption is replaced by the condition: \[
\frac{d\mu_K}{dx} \geq v \frac{d\mu_L}{dx} ~,
\]
and $D_{Che}(\Omega),D_{FM}(\Omega),D_{Exp}(\Omega)$ are replaced by
$D_{Che}(\mu_\Omega),D_{FM}(\mu_\Omega),D_{Exp}(\mu_\Omega)$ ($\Omega = K,L$) in the corresponding conclusion.
\end{lem}
\begin{proof}
Identical to the proof of the original lemmas; the only
minor point is the construction of intermediate measures $\mu_{L_i}$ in the proof of Lemma \ref{lem:going-up}, which may
be defined e.g. by $\mu_{L_i} = \frac{\eta_{L_i}}{|\eta_{L_i}|}$, $\frac{d\eta_{L_i}}{dx}(x) = \min((1+r_i)
\frac{d\mu_{L}}{dx} (\frac{x}{1+r_i}),\frac{d\mu_K}{dx}(x))$, for appropriate $r_i > 0$ (assuming the origin is in the interior of the support of $\mu_L$).
\end{proof}

The analogue of Theorem \ref{thm:Intro-stability} may then be conveniently formulated using the total-variation metric:
\[
 d_{TV}(\mu_1,\mu_2) := \frac{1}{2} \int \abs{ \frac{d\mu_1}{dx}(x) - \frac{d\mu_2}{dx}(x) } dx ~.
\]

\begin{thm} \label{thm:stability-inclusion-lc}
Let $\mu_1,\mu_2$ denote two log-concave probability measures in $(\Real^n,\abs{\cdot})$. If:
\[
d_{TV}(\mu_1,\mu_2) \leq 1 - \eps < 1 ~,
\]
then:
\[
c(\eps)^{-1} D_{Che}(\mu_2) \geq D_{Che}(\mu_1) \geq c(\eps) D_{Che}(\mu_2) ~,
\]
with $c(\eps) = c' \eps^2 / \log(1+1/\eps)$ and $c'>0$ a universal constant.
\end{thm}
\begin{proof}
Let $\mu_0$ denote the measure whose density is $\min(\frac{d\mu_1}{dx},\frac{d\mu_2}{dx})$, and note that
$d_{TV}(\mu_1,\mu_2) = 1 - |\mu_0|$. Denoting by $\mu_3$ the (log-concave) probability measure
$\frac{\mu_0}{|\mu_0|}$, since
$\frac{d\mu_i}{dx} \geq |\mu_0| \frac{d\mu_3}{dx}$, $i=1,2$, we may apply Lemma \ref{lem:gen-up-down} and the
Main Theorem to pass from $\mu_1$ to $\mu_3$ to $\mu_2$ as in
(\ref{eq:thm-inclusion-stability}), concluding the proof.
\end{proof}

\subsection{Optimality of Stability}

To the best of our knowledge, no quantitative results on the
stability of $D_{Che}$ or $D_{Poin}$ for convex domains with respect to volume
preserving perturbations or geometric distance were previously
known. Moreover, we claim that the bounds obtained in Theorem
\ref{thm:Intro-stability} (or (\ref{eq:thm-inclusion-stability})) are optimal (up to numeric
constants) with respect to $v_L$ and close to optimal with respect
to $v_K$ (note that the dependence is logarithmic in the former yet
quadratic in the latter; in other words, the deterioration in the
Cheeger constant when passing from an outer convex body to an inner
one is genuinely different than when passing from the inner one
outward). This is witnessed by the following:

\begin{ex} \label{ex:quadratic}

Let $Q^{k}$ denote a $k$-dimensional cube of volume 1, and let
$B_1^k$ denote the homothetic copy of the unit-ball of $\ell_1^k$
having volume 1.
For $2 \leq k \leq n-1$, set $K_k = Q^{n-k} \times B_1^k$ and
$L_k = Q^{n-k} \times [-c_1 k, c_1 k] \times c_2 B_1^{k-1}$, where $0<c_1,c_2<1$ are
universal constants chosen so that $L_k \subset K_k$ (it is easy to check that this is possible).
Using a tensorization result of Bobkov and Houdr\'e \cite{BobkovHoudre}, it follows that:
\begin{eqnarray*}
D_{Che}(K_k) &\simeq& \min(D_{Che}(Q^{n-k}),D_{Che}(B_1^k)) ~, \\
D_{Che}(L_k) &\simeq& \min(D_{Che}(Q^{n-k}),D_{Che}(B_1^{k-1}),D_{Che}([-k,k])) ~.
\end{eqnarray*}
It is known (see Subsection \ref{subsect:che-families}) that $D_{Che}(Q^{m}) \simeq D_{Che}(B_1^{m}) \simeq 1$, so by the $-1$-homogeneity of $D_{Che}$, it follows that $D_{Che}(K_k) \simeq 1$ and $D_{Che}(L_k) \simeq \frac{1}{k}$. Denoting $v_k = \frac{\Vol{L_k}}{\Vol{K_k}}$, since $\log 1/v_k \simeq k$, we conclude that:
\[
 D_{Che}(L_k) \simeq \frac{1}{\log(1+1/v_k)} D_{Che}(K_k) ~,
\]
uniformly for all $k=2,\ldots,n-1$. So one cannot expect better than logarithmic dependence on $1/v$
(at least when $v \geq \exp(-n)$), which coincides with the estimate given by Lemma \ref{lem:going-down}.

On the other hand (as is well-known), if we set $L = Q^n$ and $K = Q^{n-1}
\times t Q^1$ a circumscribing box with $t > 1$, since $D_{Che}(K) \simeq 1/t$ in that range, it is clear that the quadratic dependence on $v$ in Lemma \ref{lem:going-up} cannot be improved beyond linear. Although we do not know
whether the optimal bound is, up to a constant, closer to the linear
or quadratic asymptotic, we comment that for very small
perturbations (i.e. $v$ very close to 1), it is possible to show
that the exact quadratic bound in Lemma \ref{lem:going-up} \emph{is}
optimal (in this range of $v$, we of course do not allow any
additional numerical constants).

\end{ex}

The next example (which is similar yet different from the previous one) shows that the bounds in Corollary \ref{cor:isomorphism} are optimal too (up to numeric constants), as a function of $s$ in the stated range.

\begin{ex} \label{ex:isomorphism}
Continuing with the notations of Example \ref{ex:quadratic}, let us denote by $r_n$ half the diameter of $B_1^n$, so that $B_1^n = r_n \text{Conv}(\pm e_1,\ldots, \pm e_n)$, where $\text{Conv}$ denotes the convex-hull operation and $\set{e_i}$ is the standard orthonormal basis of $\Real^n$.
It is easy to check that $r_n / n \simeq 1$ uniformly on $n$. For $1 \leq s \leq c_1 n$, where $0<c_1< \frac{r_n}{2n}$
is some universal constant, define $K_s = B_1^n \cap \set{\abs{x_1} \leq s}$. It is easy to check
that in that range of $s$, $\Vol{K_s} \geq c_2 \Vol{B_1^n}$ for some universal constant $c_2>0$, and hence by
Theorem \ref{thm:Intro-stability} we deduce that $D_{Che}(K_s) \simeq 1$ uniformly on $s,n$. Now define:
\[
L_s = \text{Conv}(K_s \cap \set{x_1 = s} , K_s \cap \set{x_1 = -s}) = [-s,s] \times \brac{1 - \frac{s}{r_n}} (B_1^n \cap \set{x_1 = 0}) ~.
\]
It follows as in Example \ref{ex:quadratic} that:
\[
D_{Che}(L_s) \simeq \min\brac{D_{Che}([-s,s]), \frac{D_{Che}(B_1^n \cap \set{x_1 = 0})}{1 - \frac{s}{r_n}}} \simeq \min\brac{\frac{1}{s}, \frac{r_{n-1}}{r_n} D_{Che}(B_1^{n-1})} \simeq \frac{1}{s} ~.
\]
Since clearly $L_s \subset K_s$, it remains to note that $(1 - \frac{s}{r_n}) K_s \subset L_s$,
so $d_G(K_s,L_s) - 1 \simeq \frac{s}{n}$. By interchanging the roles of $K_s,L_s$ appropriately,
we observe that the estimates on $D_{Che}(K) / D_{Che}(L)$ in Corollary \ref{cor:isomorphism} are sharp
both from above and from below.
\end{ex}

\begin{rem}
It is easy to adapt the proofs of Lemma \ref{lem:going-down} and consequently Corollary \ref{cor:isomorphism} to
obtain even sharper quantitative bounds (up to universal constants) on the stability of $D_{Che}$ for \emph{specific} convex bodies, such as the Euclidean ball $B_2^n$. For instance, in the latter case, one obtains that if $d_G(K,B_2^n) \leq 1 + \frac{s}{n}$ for $1 \leq s \leq C_1 n$, then:
\[
 D_{Che}(K) \geq \frac{1}{C_2 \sqrt{s}} D_{Che}(B_2^n) ~.
\]
This is an improvement over Corollary \ref{cor:isomorphism} and known to be sharp for $s=n$ (folklore).
\end{rem}

\subsection{Stability of $D_{Che}$ under Lipschitz Maps}

It is well known and immediate to see that isoperimetric inequalities are preserved under
$1$-Lipschitz mappings. Given two metric probability spaces $(X,d_X,\mu)$ and $(Y,d_Y,\nu)$, recall
that a Borel map $T : (X,d_X) \rightarrow (Y,d_Y)$ is said to push forward
$\mu$ onto $\nu$, if $\nu(A) = \mu(T^{-1}(A))$ for every Borel set $A \subset Y$. This is
equivalent to requiring that for any Borel function $g$ on $(Y,d_Y)$:
\[
 \int_Y g(y) d\nu(y) = \int_X g(T(x)) d\mu(x) ~.
\]
This will be denoted by $T_*(\mu) = \nu$. The following is then immediate from the definitions:

\begin{fact*}
Assume that $T_*(\mu) = \nu$. Then:
\[
 I_{(Y,d_Y,\nu)} \geq \frac{1}{\norm{T}_{Lip}} I_{(X,d_X,\mu)} ~.
\]
\end{fact*}

Here as usual:
\[
\norm{T}_{Lip} := \sup_{x \neq y \in X} \frac{d_Y(T(x),T(y))}{d_X(x,y)} ~.
\]

The following result states that when our convexity assumptions hold for the target space, as far
as Cheeger's isoperimetric inequality is concerned, one need not require that $T$ be Lipschitz on
the entire space, but rather just on average. We would like to thank Bo'az Klartag for a fruitful discussion
regarding this point.

\begin{thm} \label{thm:Lip-on-average}
Assume that $(Y,d_Y,\nu)$ verifies our convexity assumptions and that $T_*(\mu) = \nu$ for some
Lipschitz-on-balls map $T$. Then:
\[
 D_{Che}(Y,d_Y,\nu) \geq \frac{c}{\int_X \norm{D T}_{op}(x) d\mu(x)} D_{Che}(X,d_X,\mu) ~,
\]
for some universal constant $c>0$.
\end{thm}

Here $\norm{D T}_{op}(x)$ denotes the local Lipschitz constant of $T$ at $x$:
\[
 \norm{D T}_{op}(x) := \limsup_{y \rightarrow x} \frac{d_Y(T(x),T(y))}{d_X(x,y)} ~.
\]
When $T$ is smooth and $X,Y$ are linear spaces, this coincides with the operator norm of the
usual derivative matrix $D T$ at $x$.

\begin{proof}
First, rewrite Cheeger's isoperimetric inequality on $(X,d_X,\mu)$ in functional form (Lemma
\ref{lem:Cheeger=L1}):
\begin{equation} \label{eq:Lip-ass}
\forall f \in \F(X,d_X) \;\;\; D_{Che}(X,d_X,\mu) \norm{f - M_\mu f}_{L_1(X,\mu)} \leq
\norm{\abs{\nabla_X
f}}_{L_1(X,\mu)} ~.
\end{equation}
Using this, we estimate the First-Moment constant on $(Y,d_Y,\nu)$. Given a $1$-Lipschitz function
$g$ on $(Y,d_Y)$, clearly $g \circ T$ is Lipschitz-on-balls on $(X,d_X)$, hence in $\F(X,d_X)$.
We then have by the definition of push-forward and our assumption (\ref{eq:Lip-ass}):
\begin{eqnarray*}
& & \int_Y \abs{g - M_\nu g} d\nu = \int_X \abs{g(T x) - M_\mu (g \circ T)} d\mu \\
&\leq& \frac{1}{D_{Che}(X,d_X,\mu)} \int_X \abs{\nabla_X (g\circ T)}(x) d\mu(x) \\
&\leq& \frac{1}{D_{Che}(X,d_X,\mu)} \int_X \norm{D T}_{op}(x) \abs{\nabla_Y g}(Tx) d\mu(x)
\leq \frac{\int_X \norm{D T}_{op}(x) d\mu(x)}{D_{Che}(X,d_X,\mu)} ~.
\end{eqnarray*}
Hence:
\[
 D^M_{FM}(Y,d_Y,\nu) \geq \frac{D_{Che}(X,d_X,\mu)}{\int_X \norm{D T (x)}_{op} d\mu(x)} ~.
\]
We conclude by our Main Theorem (and Lemma \ref{lem:E-M}), which imply that $D_{Che}(Y,d_Y,\nu) \geq c
D^M_{FM}(Y,d_Y,\nu)$ under our convexity assumptions on $(Y,d_Y,\nu)$.
\end{proof}

\subsection{Estimating $D_{Che}$}

In this subsection, we easily recover some previously known
estimates on the Cheeger constant of convex domains in a single
framework and extend some results to the Riemannian setting.
We begin with the following stimulating conjecture from \cite{KLS}:

\begin{conj*}[Kannan--Lov\'asz--Simonovits]
There exists a universal constant $c>0$ such that for any convex body $K$ in
$(\Real^n,\abs{\cdot})$, and more generally, for any log-concave probability measure $\mu$ on
$(\Real^n,\abs{\cdot})$:
\begin{equation} \label{eq:KLS-conj}
 D_{Che}(\mu) \geq \frac{c}{\sigma_1(\mu)} ~.
\end{equation}
\end{conj*}
\noindent
Here $\sigma_1(\mu)^2$ denotes the largest eigenvalue of the
symmetric covariance matrix $\Sigma(\mu)$ of $\mu$:
\[
 \Sigma(\mu) := E_\mu(x \otimes x) - E_\mu(x) \otimes E_\mu(x) ~.
\]
We will write $\sigma_1(K)$ for $\sigma_1(\lambda_K)$.

\medskip

Standard results in Convexity Theory easily imply that the opposite
inequality in (\ref{eq:KLS-conj}) holds with some universal constant
$c>0$. The reason for this is that it is easy to analyze the
isoperimetric inequality for sets which are half-spaces in
$\Real^n$, and when restricting to these sets, both the upper bound
and the conjectured lower bound hold with some (explicitly known)
universal constants. The KLS conjecture is therefore a striking
statement on the nature of isoperimetric minimizing sets for
Cheeger's isoperimetric inequality in the convex setting: these sets
do not minimize boundary-measure much better than just half-spaces.
An explicit description of the isoperimetric minimizers is known
only in a few cases, even in the Euclidean setting
$(\Omega,\abs{\cdot},\lambda_\Omega)$ (see e.g.
\cite{RosIsoperimetricProblemNotes}),
so it is extremely important to at least identify some
\emph{essentially} minimizing sets (up to universal constants).

\medskip

Although the KLS conjecture is far from being resolved, some general
lower bounds on $D_{Che}$ are known, but these produce
dimension-dependent results. We will see that our Main Theorem
easily reproduces these bounds.

\smallskip

The following result in the Euclidean setting is due to Payne and Weinberger \cite{PayneWeinberger}.
This was generalized to the Riemannian setting by Li and Yau \cite{LiYauEigenvalues}. We refer to
the Appendix for missing definitions.

\begin{thm}[Payne--Weinberger, Li--Yau] \label{thm:Yau-Li}
If $K \subset (M,g)$ is a locally convex bounded domain with smooth boundary and $Ric_g \geq 0$,
then:
\[
 D_{Poin}(K,d_g,\lambda_K) \geq \frac{\pi}{2 \diam(K)} ~,
\]
where $\diam$ denotes the diameter and $d_g$ the induced geodesic distance.
In fact, when $(M,g)$ is Euclidean space the constant 2 above may be omitted.
\end{thm}

Ledoux's Theorem \ref{thm:Ledoux} implies that the same lower bound
(up to an additional constant) holds for $D_{Che}(K,d_g,\lambda_K)$.
In the Euclidean case, this was strengthened in \cite{KLS}:

\begin{thm}[Kannan--Lov\'asz--Simonovits] \label{thm:KLS}
Let $\mu$ be a log-concave probability measure on
$(\Real^n,\abs{\cdot})$. Then:
\[
 D_{Che}(\mu) \geq \sup_{x_0 \in \Real^n} \frac{c}{\int \abs{x - x_0} d\mu(x)} ~,
\]
for some universal constant $c>0$.
\end{thm}

To obtain this result, Kannan, Lov\'asz and Simonovits developed a
geometric localization technique (which in fact can be traced back
to the work of M. Gromov and V. Milman \cite{Gromov-Milman}). As
pointed out to us by Sasha Sodin, it is interesting to note that
this technique uses some geometric properties of Euclidean space and
does not generalize to other Riemannian manifolds (except in special
cases, like that of the Euclidean Sphere, as in the work of
Gromov--Milman). Our method, on the other hand, does allow us to
state the following generalization of Theorem \ref{thm:KLS} to the
Riemannian setting, which also improves over Theorem \ref{thm:Yau-Li}:

\begin{thm} \label{thm:My-KLS}
Assume that $(\Omega,d,\mu)$ satisfies our convexity assumptions.
Then:
\[
 D_{Che}(\Omega,d,\mu) \geq \sup_{x_0 \in \Omega} \frac{c}{\int d(x,x_0) d\mu(x)} ~,
\]
for some universal constant $c>0$.
\end{thm}
\begin{proof}
As usual, we just need to bound $D_{FM}(\Omega,d,\mu)$. Let $f$
denote a $1$-Lipschitz function on $(\Omega,d)$. Then for any $x_0
\in \Omega$, applying the triangle inequality twice:
\begin{eqnarray*}
 & &\int \abs{f(x) - E_\mu f} d\mu(x) \leq \int \abs{f(x) - f(x_0)} d\mu(x) + \abs{E_\mu f - f(x_0)}
\\ & \leq&  2 \int \abs{f(x) - f(x_0)} d\mu(x) \leq 2 \int d(x,x_0) d\mu(x) ~.
\end{eqnarray*}
Hence:
\[
 D_{FM}(\Omega,d,\mu) \geq \sup_{x_0 \in \Omega} \frac{1}{2 \int d(x,x_0) d\mu(x) } ~,
\]
and the claim follows by our Main Theorem.
\end{proof}

\begin{rem}
An alternative approach to localization for proving isoperimetric inequalities was developed by Bobkov
\cite{BobkovGaussianIsoLogSobEquivalent} in the Euclidean setting. Bobkov's approach was extended by Barthe
\cite{BartheIntegrabilityImpliesIsoperimetryLikeBobkov} and subsequently by Barthe and Kolesnikov \cite{BartheKolesnikov}.
This approach is based on the Pr\'ekopa--Leindler inequality (e.g. \cite{BrascampLiebPLandLambda1}),
or equivalently, on optimal transportation, which have both been recently generalized to the Riemannian-with-density-setting by
Cordero-Erausquin, McCann and Schmuckenschl{\"a}ger \cite{CMSInventiones,CMSManifoldWithDensity}. Using these tools we
expect that it should be possible to provide an alternative proof of Theorem
\ref{thm:My-KLS} following Bobkov's approach, but as pointed out to us by one of the referees, this has yet to be accomplished.
We would like to thank the
referee for his comments regarding our original simpleminded remark in this direction.
\end{rem}

We would like to mention another bound on $D_{Che}$ obtained in \cite{KLS} using the localization
method.

\begin{thm}[Kannan--Lov\'asz--Simonovits] \label{thm:KLS2}
Let $\mu$ be a log-concave probability measure on $(\Real^n,\abs{\cdot})$ with bounded support
$B$. Then:
\[
 D_{Che}(\mu) \geq \frac{c}{\int \theta_B(x) d\mu} ~,
\]
where $\theta_B(x)$ denotes the longest symmetric interval contained
in $B$ and centered at $x$, and $c>0$ is a universal constant.
\end{thm}

We have recently managed to derive this result using our Main
Theorem, but this will be described elsewhere. Instead, we would like to show
how this bound may be used to
recover a result of Bobkov \cite{BobkovVarianceBound};
in fact, the bound we deduce is formally stronger than Bobkov's.
Bobkov employs the localization method as well, but then relies on
some nice trick involving moment inequalities for polynomials in the
log-concave setting. Our argument, on the other hand, is more
geometric. Independently of our proof, we heard about a similar idea
for bounding the boundary measure of large sets from Santosh Vempala
(using localization as well).

\begin{thm}[Bobkov] \label{thm:Bobkov}
Let $\mu$ be a log-concave probability measure on $(\Real^n,\abs{\cdot})$. Then:
\[
 D_{Che}(\mu) \geq \sup_{x_0 \in \Real^n} \frac{c}{(Var_\mu(|x-x_0|^2))^{1/4}} ~,
\]
where $Var_\mu$ denotes the variance with respect to $\mu$.
\end{thm}
\begin{proof}[Sketch of Proof]
Without loss of generality, we may assume that $x_0=0$; for general $x_0$ the
claimed bound follows by translating $\mu$. Let $E:= E_\mu |x|$, $S  := (Var_\mu |x|)^{1/2}$, and denote: \[
B := \set{x \in \Real^n ; |x| \leq E + 2 S} ~.
\]
By Chebyshev's inequality,
$\mu(B) \geq 3/4$, so if we define $\mu_{0} := \mu|_{B} / \mu(B)$, it follows
that $d_{TV}(\mu,\mu_{0}) \leq 1/4$. Hence $D_{Che}(\mu) \simeq D_{Che}(\mu_{0})$ by
Theorem \ref{thm:stability-inclusion-lc}. Assume that $E \geq 2S$, otherwise the support of $\mu_0$ has
diameter bounded by $8 S$, and one can conclude as in Theorem \ref{thm:My-KLS}. We now employ Theorem \ref{thm:KLS2}
to bound $D_{Che}(\mu_0)$:
\begin{equation} \label{eq:Bobkov-proof}
D_{Che}(\mu_0) \geq \frac{c}{\int \theta_B(x) d\mu_0(x)} = \frac{c \mu(B)}{\int_B \theta_B(x)
d\mu(x)} ~.
\end{equation}
The crucial geometric observation is that for the Euclidean ball $B$:
\[
 \theta_B(x) = 2 \sqrt{(E+2 S)^2 - |x|^2} ~.
\]
It remains to plug this into (\ref{eq:Bobkov-proof}) and evaluate the resulting expression using integration
by parts and Chebyshev's inequality. We leave it as an exercise to conclude that:
\[
D_{Che}(\mu) \geq \frac{c'}{\sqrt{ES}} ~,
\]
for some universal constant $c'>0$. This bound is in
fact formally better than Bobkov's bound (by several applications of H\"{o}lder's inequality), but
using some standard results in Convexity Theory, it is in fact equivalent in the interesting
situations.
\end{proof}

\subsection{$D_{Che}$ for Specific Families of Convex Bodies} \label{subsect:che-families}

Embarrassingly, hardly any concrete examples exist of non-degenerate
convex bodies $K$ in $\Real^n$ for which the asymptotic value of
$D_{Che}(K)$ (as a function of the dimension $n$) is known. The KLS conjecture stating that $D_{Che}(K) \simeq 1$ for such
bodies has only been confirmed in a few special cases.
These include the Euclidean ball (see e.g. \cite{BuragoZalgallerBook}) and the unit cube
$K = [-1/2,1/2]^{n}$  (Hadwiger \cite{HadwigerCube}, see also \cite{BollobasLeaderCube},\cite{BartheMaureyIsoperimetricInqs}).
By the tensorization results of Bobkov and Houdr\'e \cite{BobkovHoudre}, this is in fact
true for an arbitrary log-concave product measure (appropriately normalized).
When $K = \tilde{B}(\ell_p^n)$, the volume one homothetic copy of
the unit-ball of $\ell_p^n$, for $p\in [1,2]$, the KLS conjecture was only recently
confirmed by Sasha Sodin \cite{SodinLpIsoperimetry} (note that
indeed $\sigma_1(\tilde{B}(\ell_p^n)) \simeq 1$). Even more recently, the case $p \geq 2$
has been confirmed by R. Lata{\l}a and J. Wojtaszczyk \cite{LatalaJacobInfConvolution} by an elegant construction of a Lipschitz map pushing forward the Gaussian measure onto the uniform measure on $\tilde{B}(\ell_p^n)$. We are not aware of any other (sufficiently different) examples.

We comment that our Main Theorem easily implies the result for $K = \tilde{B}(\ell_p^n)$, $p \in
[1,2]$, due to Sodin \cite{SodinLpIsoperimetry}. However, Sodin's result provides a sharp bound on
the isoperimetric profile of these spaces, whereas we only deduce the bound on Cheeger's constant.

\begin{thm}[Sodin]
For any $n \geq 1$, $p \in [1,2]$:
\[
 D_{Che}(\tilde{B}(\ell_p^n)) \geq c > 0 ~,
\]
where $c>0$ is a universal constant.
\end{thm}
\begin{proof}
This is immediate from the results of Schechtman and Zinn \cite{SchechtmanZinn2}, who showed that
$D_{Exp}$ of these bodies is bounded from below by a universal constant. The result then follows
from our Main Theorem (in fact, we only need a bound on $D_{FM}$).
\end{proof}

Another family of convex bodies for which the KLS conjecture is
almost confirmed, is that of unconditional convex bodies $K$, i.e.
convex bodies for which $(x_1,\ldots,x_n) \in K$ iff $(\pm x_1,
\ldots, \pm x_n) \in K$. It was recently shown by Bo'az Klartag
\cite{KlartagUnconditionalVariance} that if $K$ is an unconditional
body with $\sigma_1(K) = 1$ then $D_{Che}(K) \geq c / \log n$, for
some universal constant $c>0$. To obtain this result, Klartag
employed Theorem \ref{thm:Intro-stability} to pass to an
unconditional body contained inside the cube $(C \log n) [-1,1]^n$,
and then used some symmetry properties of the Laplacian's
eigenfunctions to conclude his result. In fact, one can just use
Theorem \ref{thm:Intro-concavity} on the concavity of the
isoperimetric profile (in the form of Lemma \ref{lem:going-up}) for
this application.

\subsection{Some dimension dependent bounds on $D_{Che}$}

We conclude this section by stating the known dimension dependent
bounds on $D_{Che}(K)$ for non-degenerate convex bodies $K$ (in the
sense that $\sigma_1(K) = 1$).

It is known in this case that $\diam(K) \leq c n$ (by a simple volume estimate). Theorem \ref{thm:Yau-Li} (together with Theorem \ref{thm:Ledoux}) then gives $D_{Che}(K) \geq c / n$. The first KLS bound
(Theorem \ref{thm:KLS}) improves this to $D_{Che}(K) \geq c / \sqrt{n}$, since:
\[
 \int_K \abs{x - E_\mu x} dx \leq (\int_K \abs{x - E_\mu x}^2 dx)^{1/2} \leq \sqrt{n} \sigma_1(K)
~.
\]
The second KLS bound (Theorem \ref{thm:KLS2}) is incomparable to the
first bound, since it gives the right order for the Euclidean ball,
but gives $c / n$ for the regular simplex of volume 1 in $\Real^n$.

Bobkov's bound (Theorem \ref{thm:Bobkov}) is always at least as good
as the first KLS bound (up to a constant), since (using the bound
derived in the proof together with a standard application of
Borell's lemma \cite{Borell-logconcave}):
\[
 Var_\mu(\abs{x - x_0})^{1/2} \leq E_\mu(\abs{x-x_0}^2)^{1/2} \leq C E_\mu(\abs{x-x_0}) ~,
\]
for some universal constant $C>0$. We see that whenever some non-trivial information on
$Var_\mu(\abs{x - x_0})$ is known, Bobkov's bound is strictly better. Such a remarkable
result was proved by Bo'az Klartag \cite{KlartagCLP,KlartagCLPpolynomial}, allowing him to
deduce a Central-Limit type result for the class of convex bodies (and more generally, log-concave
measures). Klartag's improved estimate in \cite{KlartagCLPpolynomial} reads:
\[
 Var_\mu (\abs{x - E_\mu x})^{1/2} \leq C_\eps n^{1/2 - 1/10 + \eps} \sigma_1(\mu) \;\;\; \forall
\eps > 0 ~.
\]
Combining this with Bobkov's bound, one deduces the following
result, already noticed among specialists, for log-concave measures
in $\Real^n$ with $\sigma_1(\mu) = 1$:
\[
 D_{Che}(\mu) \geq \frac{c_\eps}{n^{1/2 - 1/20 + \eps}} \;\;\; \forall \eps > 0 ~.
\]
At the moment, this is the best known bound on Cheeger's constant for general log-concave
measures (or convex bodies) in $\Real^n$.

\section{Approximation Argument} \label{sec:AA}

In this section, we develop an approximation argument for extending the following theorems
to non-necessarily smooth densities (or boundaries) in our convexity assumptions:
\begin{itemize}
\item Theorem \ref{thm:Intro-concavity} on the concavity of the isoperimetric profile.
\item Our Main Theorem \ref{thm:Main}.
\end{itemize}
We will develop different procedures for extending each of these theorems.

\subsection{Stability of the Isoperimetric Profile}

We begin by extending our definition of smooth convexity assumptions (we refer to the Appendix for the
definition of \emph{locally convex}).

\begin{dfn*}
We will say that our \emph{generalized smooth convexity assumptions} are fulfilled if:
\begin{itemize}
\item
$(M,g)$ denotes an $n$-dimensional ($n\geq 2$) smooth complete oriented connected Riemannian manifold
or $(M,g)=(\Real,\abs{\cdot})$.
\item
$\Omega \subset M$ is a locally convex domain with $C^2$ boundary.
\item
$d$ denotes the induced geodesic distance on $(M,g)$.
\item
$d\mu = \exp(-\psi)  dvol_M|_{\Omega}$, $\psi \in C^2(\overline{\Omega})$, and as tensor
fields on $\Omega$:
\[
Ric_g + Hess_g \psi \geq 0 ~.
\]
\end{itemize}
\end{dfn*}

This definition was already used in the statement of Theorem \ref{thm:Intro-concavity} on the concavity of the isoperimetric profile.
The smoothness assumptions in the above definition are used in an essential way in the proof of this theorem
to deduce the existence and regularity of the isoperimetric minimizers, which are otherwise false.
This permits the use of variational methods from Riemannian Geometry, consequently obtaining a second-order
differential inequality which the isoperimetric profile must satisfy (see the Appendix for more details).
Nevertheless, the restriction to smooth densities and domains still seems like a technical artifact of the proofs.
Some authors have suggested various methods to remove these smoothness assumptions
(see e.g. Morgan \cite{MorganLevyGromovForConvexManifolds} and Bayle \cite[Chapter 4]{BayleThesis}),
but unfortunately these are not well suited for our purposes.
We therefore attempt to use a different approximation argument for extending Theorem \ref{thm:Intro-concavity}
to a more general setting.

\medskip

At first glance, it is tempting to believe that the isoperimetric
profile of $(\Omega,d,\mu)$ should be stable under approximating the
measure $\mu$ by measures $\mu_m$ in, say, total-variation distance.
However, the profile is in fact not even pointwise continuous under arbitrary
approximation in total-variation. To see this,
consider the measures $\mu_m$ which are uniform on the set $[0,1] \setminus [1/2-1/m,1/2+1/m]$,
and converge to $\mu$, the uniform measure on $[0,1]$. Clearly $I_{\mu_m}(1/2) = 0$ for every $m \geq 3$,
even though $I_\mu(1/2) = 1$. So one must take care when specifying the approximation.

\begin{dfn*}
We say that a sequence of Borel probability measures $\set{\mu_m}$ tends to $\mu$ from above if
$\set{\mu_m}$ converges to $\mu$ in total-variation and in addition there exists a sequence
$\set{c_m}$ which tends to 1, so that $\mu_m(A) \geq \mu(A) / c_m$ for any Borel set $A$.
\end{dfn*}

\begin{lem} \label{lem:above}
Let $(\Omega,d)$ be a metric space and let $\set{\mu_m}$ be a sequence of Borel
probability measures on $(\Omega,d)$ which tends to $\mu$ from above. Then for
any $t\in(0,1)$:
\[
\liminf_{m \rightarrow \infty} I_{(\Omega,d,\mu_m)}(t) \geq \liminf_{s \rightarrow t}
I_{(\Omega,d,\mu)}(s) ~.
\]
\end{lem}
\begin{proof}
Denote $I = I_{(\Omega,d,\mu)}$ and $I_m = I_{(\Omega,d,\mu_m)}$ for short.
Let $\eps>0$. Then there exists $m_0$ such that for all $m \geq m_0$, $\abs{\mu(B) - \mu_m(B)} <
\eps$ for any Borel set $B$. Let $\delta>0$, then for every $m \geq m_0$ there exist a Borel set
$B_m$ such that:
\[
 I_m(t) + \delta \geq \mu_m^+(B_m) \geq \mu^+(B_m) / c_m \geq I(\mu(B_m)) / c_m \geq
\inf_{\abs{s-t}<\eps} I(s) / c_m ~.
\]
Taking the limit as $m \rightarrow \infty$ and subsequently $\eps,\delta \rightarrow 0$, we obtain
the assertion.
\end{proof}

\begin{dfn*}
We say that a sequence of Borel probability measures $\set{\mu_m}$ tends to $\mu$ from within if
$\mu_m = \mu |_{A_m} / \mu(A_m)$ for some sequence of Borel sets $A_m$ such that $\mu(A_m)
\rightarrow 1$, and in addition $\mu^+(A_m) \rightarrow 0$.
\end{dfn*}

\begin{lem} \label{lem:within}
Let $(\Omega,d)$ be a metric space and let $\set{\mu_m}$ be a sequence of Borel
probability measures on $(\Omega,d)$ which tends to $\mu$ from within. Then for
any $t\in(0,1)$:
\[
\liminf_{m \rightarrow \infty} I_{(\Omega,d,\mu_m)}(t) \geq \liminf_{s \rightarrow t}
I_{(\Omega,d,\mu)}(s) ~.
\]
\end{lem}
\begin{proof}
We continue with the same assumptions and notations as in the proof of the previous lemma and
definition. In our case, we may assume that $B_m \subset A_m$. Then:
\[
 I_m(t) + \delta \geq \mu_m^+(B_m) \geq \frac{\mu^+(B_m) - \mu^+(A_m)}{\mu(A_m)}
\geq \frac{I(\mu(B_m)) - \mu^+(A_m)}{\mu(A_m)} \geq
\inf_{\abs{s-t}<\eps} \frac{I(s) - \mu^+(A_m)}{\mu(A_m)} ~.
\]
Taking the limit as $m \rightarrow \infty$ and subsequently $\eps,\delta \rightarrow 0$, we obtain
the assertion.
\end{proof}

\begin{rem}
It is quite non-trivial to come up with other conditions which
ensure the conclusion of Lemmas \ref{lem:above} and
\ref{lem:within}. Of course convergence in the $L_\infty$ norm of
the densities with respect to the Riemannian volume form would also
do, but this seems an impractical assumption since $\mu$ may have a
non-continuous density. Another interesting possibility which works is to assume that $\mu_m$ are
obtained by pushing $\mu$ forward using mappings $T_m$, so that $\norm{T_m}_{Lip}$ tends to 1.
Unfortunately, we do not know how to show that an arbitrary log-concave measure
$\mu$ in $\Real^n$ may be approximated by smooth log-concave measures $\mu_m$ of this type.
\end{rem}

Next, we recall the definition of $q$-capacity (we will only require
the case $q=1$). Capacities were introduced in the 1960's by Maz'ya
\cite{MazyaSobolevImbedding,MazyaCapacities}, Federer and
Fleming \cite{FedererFleming}, and were used by Bobkov and Houdr\'e
in \cite{BobkovHoudreMemoirs,BobkovHoudre}. We follow a variation on the definition given
in \cite{MazyaBook} (for general $q$), which was
extended by Barthe, Cattiaux and Roberto (with $q=2$) in
\cite{BCRHard} (after being introduced in \cite{BartheRoberto}). We
conform to the definition implicitly used by Sodin in
\cite{SodinLpIsoperimetry} and Sodin and the author in
\cite{EMilmanSodinIsoperimetryForULC}.

\begin{dfn*}
Given a metric probability space $(\Omega,d,\mu)$, $0<q<\infty$ and
$0 \leq a \leq b \leq 1$, we denote:
\[
 \Cap_q(a,b) := \inf\set{ \norm{\abs{\nabla \Phi}}_{L_q(\mu)} ; \mu\set{\Phi=1} \geq a \;,\;
\mu\set{\Phi=0} \geq 1-b },
\]
where the infimum is on all $\Phi : \Omega \rightarrow [0,1]$ which
are Lipschitz-on-balls (recall the definition of $\abs{\nabla \Phi}$
given in Remark \ref{rem:grad-defn}).
\end{dfn*}

The following proposition encapsulates the connection between
$1$-capacity and the isoperimetric profile $I = I_{(\Omega,d,\mu)}$.
The proof is very much along the lines of the proof of Lemma
\ref{lem:Cheeger=L1}, so we will omit it here; the reader is referred to
Sodin \cite[Proposition A]{SodinLpIsoperimetry} for an elementary
derivation (note the slight difference in our formulation). We only
remark that it suffices to use Lipschitz functions $\Phi$ in the
definition of capacity above for the purpose of this proposition.

\begin{prop}[Maz'ya, Federer--Fleming, Bobkov--Houdr\'e] \label{prop:Sodin-prop}
For all $0<a < b < 1$:
\begin{equation} \label{eq:Sodin-prop}
 \inf_{a \leq t \leq b} I(t) \leq \Cap_1(a,b) \leq \inf_{a \leq t < b}
I(t) ~.
\end{equation}
\end{prop}

Since obviously $\Cap_1(a,b) = \Cap_1(1-b,1-a)$, it follows that:
\[
 \inf_{a \leq t \leq b} I(t) \leq \inf_{1-b \leq t < 1-a} I(t) ~.
\]
Letting $b$ converge to $a$, and replacing $a,b$ with $1-b,1-a$, we
obtain:
\begin{cor} \label{cor:profile-symmetry}
If $I$ is lower semi-continuous at $t$ and $1-t$, $t \in (0,1)$, then $I(t) =
I(1-t)$.
\end{cor}

\begin{lem} \label{lem:approx-easy-direction}
Let $(\Omega,d)$ be a metric space and let $\set{\mu_m}$ be a sequence of Borel
probability measures on $(\Omega,d)$ which converges in the total-variation norm to $\mu$.
Assume in addition that $I_{(\Omega,d,\mu_m)}$ are concave on $(0,1)$. Then for any $t \in (0,1)$:
\[
\liminf_{s\rightarrow t} I_{(\Omega,d,\mu)}(s) \geq \limsup_{m
\rightarrow \infty} I_{(\Omega,d,\mu_m)}(t) ~.
\]
\end{lem}
\begin{proof}
As usual, denote $I = I_{(\Omega,d,\mu)}$ and $I_m =
I_{(\Omega,d,\mu_m)}$ for short. Let $t\in (0,1)$ and small $\eps>0$
be given, and let $\Phi : (\Omega,d) \rightarrow [0,1]$ denote a
Lipschitz function so that:
\[
\mu\set{\Phi=1} \geq t-\eps \;\;,\;\; \mu\set{\Phi=0} \geq 1-t-\eps
~.
\]
For any small $\delta>0$, there exists an $m_0$ so that for any $m \geq m_0$:
\[
\mu_m\set{\Phi=1} \geq t-\eps-\delta \;\;,\;\; \mu_m\set{\Phi=0}
\geq 1-t-\eps-\delta ~.
\]
We conclude by Proposition \ref{prop:Sodin-prop} and the concavity of $I_m$ that:
\[
 \int |\nabla \Phi| d\mu_m \geq \inf_{t - \eps- \delta \leq s \leq t+\eps+\delta} I_m(s) \geq
\min\brac{\frac{t-\eps-\delta}{t},\frac{1-t-\eps-\delta}{1-t}}
I_m(t) ~.
\]
Since $\Phi$ is Lipschitz (hence $|\nabla \Phi|$ is bounded), and
$\set{\mu_m}$ converge to $\mu$ in total-variation, we can pass to the
limit as $m \rightarrow \infty$:
\[
 \int |\nabla \Phi| d\mu \geq
\min\brac{\frac{t-\eps-\delta}{t},\frac{1-t-\eps-\delta}{1-t}}
\limsup_{m\rightarrow \infty} I_m(t) ~.
\]
Taking infimum on all such $\Phi$ as above and using Proposition
\ref{prop:Sodin-prop} again, we obtain:
\[
 \inf_{t-\eps \leq s < t+\eps} I(s) \geq
\min\brac{\frac{t-\eps-\delta}{t},\frac{1-t-\eps-\delta}{1-t}}
\limsup_{m\rightarrow \infty} I_m(t) ~.
\]
Taking the limit of $\eps,\delta$ to 0, we obtain the desired conclusion.
\end{proof}

\begin{rem}
It is clear from the proof that the concavity condition may be seriously relaxed (e.g. to
equicontinuity), and the regularity condition on $I_m$ obtained in Lemma \ref{lem:I-continuity}
below may also be used.
\end{rem}

Combining the last three lemmas we immediately obtain:

\begin{prop} \label{prop:convergence}
Let $(\Omega,d)$ be a metric space, let $\set{\mu_m}$ be a sequence of Borel
probability measures on $(\Omega,d)$ which converges in the total-variation norm to $\mu$, and
assume that $I_{(\Omega,d,\mu_m)}$ are all concave on $(0,1)$. If in addition $\set{\mu_m}$ tend to $\mu$
from above or from within, then for any $t\in (0,1)$:
\[
\liminf_{m \rightarrow \infty} I_{(\Omega,d,\mu_m)}(t)=
\limsup_{m \rightarrow \infty} I_{(\Omega,d,\mu_m)}(t)=
\liminf_{s\rightarrow t} I_{(\Omega,d,\mu)}(s) ~.
\]
In particular, if $I_{(\Omega,d,\mu)}$ is in addition lower semi-continuous, we have (pointwise):
\[
\lim_{m \rightarrow \infty} I_{(\Omega,d,\mu_m)}= I_{(\Omega,d,\mu)} ~.
\]
\end{prop}

The following lemma, which extends the argument given by Gallot in
\cite[Lemma 6.2]{GallotIsoperimetricInqs} for compact manifolds with uniform density, provides a
sufficient condition for the isoperimetric profile to be continuous.

\begin{lem} \label{lem:I-continuity}
Let $\Omega = (M,g)$ denote an $n$-dimensional ($n\geq 2$) smooth complete oriented connected Riemannian manifold and
let $d$ denote the induced geodesic distance. Let $\mu$ denote an absolutely continuous measure with
respect to $vol_M$, such that its density is bounded from above on every ball (but not necessarily from below,
nor do we assume it is continuous). Then $I = I_{(\Omega,d,\mu)}$ is absolutely continuous on
$[0,1]$, and in fact is locally of H\"{o}lder exponent $\frac{n-1}{n}$.
\end{lem}

\begin{proof}
By Lebesgue's Theorem, we know for almost every $x\in M$ (with respect to
$vol_M$),
\[
\mu(B_M(x,\eps)) = \frac{d\mu}{dvol_M}(x) \tVol_M(B_M(x,\eps)) (1 +
o(1)) ~,
\]
and clearly:
\begin{equation} \label{eq:sa-estimate}
\mu^+(B_M(x,\eps)) \leq \mu_\infty(\overline{B_M(x,2\eps)}) \tVol_M(\partial
B_M(x,\eps)) ~,
\end{equation}
where $B_M(x,R)$ denotes the ball in $M$ of radius $R$ around $x$, $\tVol_M$ denotes the
Riemannian volume on $M$ (and by abuse of notation the induced volume on any submanifold as well),
and $\mu_\infty(\C)$ denotes the upper bound on the density of $\mu$ on a compact set $\C \subset M$.
By Rauch's Comparison Theorem, for any such compact set $\C$ (and in particular a singleton),
there exists a $\eps_\C < 1/2$ so that for any $x \in \C$ and $\eps < \eps_\C$:
\begin{equation} \label{eq:control-minus-1}
 \frac{3}{4} \eps^n \Vol{B^n} < \tVol_M(B_M(x,\eps)) < \frac{5}{4} \eps^{n} \Vol{B^n} \; ,
\end{equation}
\begin{equation} \label{eq:control0}
 \frac{3}{4} \eps^{n-1} \Vol{S^{n-1}} < \tVol_M(\partial B_M(x,\eps)) < \frac{5}{4} \eps^{n-1}
\Vol{S^{n-1}} \; ,
\end{equation}
where $B^n$ and $S^{n-1}$ denote the Euclidean unit ball and sphere, respectively, and
$\tVol$ denotes Euclidean volume. Therefore as $t\rightarrow 0$:
\[
I(t),I(1-t) \leq C_{n,\mu} t^{(n-1)/ n} (1+o(1)) ~,
\]
where $C_{n,\mu}$ depends on $n$ and $\mu$ only. Since clearly $I(0) = I(1) = 0$, this takes care of
the continuity at $0$ and $1$.

Now fix $x_0 \in M$ and define $g : (0,1) \rightarrow \Real_+$ to be the function:
\[
 g(\eps) := \inf\set{ R > 0 ; \mu(B_M(x_0,R)) \geq 1-\eps} ~.
\]
Given $0<\theta<1$, set $R_\theta = g(\theta/2) + 1$, $\eps_\theta = \eps_{B_M(x_0,R_\theta+1)}$, and
$\mu_\infty(\theta) = \mu_\infty(\overline{B_M(x_0,R_\theta+1)})$. Let $K_\theta$ denote the
(possibly negative) lower bound on the sectional curvature of $K$ on $B_M(x_0,R_\theta)$.
Rauch's Theorem also implies that:
\begin{equation} \label{eq:control1}
\tVol_M (B_M(x_0,R_\theta)) \leq \tVol_{M_{K_\theta}}(B_{M_{K_\theta}}(R_\theta)) ~,
\end{equation}
where $M_K$ denotes the simply connected model space with constant curvature $K$, $\tVol_{M_K}$
denotes the volume on $M_K$ and $B_{M_K}(R)$ is any ball in $M_K$ of radius $R$.

Given a set $A \subset M$ with $\theta = \mu(A)>0$, note that by Fubini's
Theorem, (\ref{eq:control-minus-1}) and the definition of $g$, for any $\eps < \eps_\theta < 1$:
\begin{eqnarray}
\nonumber
 & & \int_{B_M(x_0,R_\theta)} \mu(A \cap B_M(x,\eps)) dvol_M(x) = \int_A \tVol_M (B_M(y,\eps) \cap
B_M(x_0,R_\theta)) d\mu (y) \\
\nonumber &\geq& \int_{A \cap B_M(x_0,R_\theta-1)} \tVol_M (B_M(y,\eps)) d\mu(y) \geq \frac{3}{4} \eps^n
\Vol{B^n} \mu(A \cap B_M(x_0,g(\mu(A)/2))) \\
\label{eq:control2} &\geq& \frac{3}{8} \eps^n \Vol{B^n} \mu(A) ~.
\end{eqnarray}
We conclude from $(\ref{eq:control2})$ and $(\ref{eq:control1})$
that given any $A \subset M$ with $0<\theta=\mu(A)<1$ and $\eps<\eps_\theta$, there exists an
$x\in B_M(x_0,R_\theta)$ such that:
\begin{equation} \label{eq:control3}
\mu(A \cap B_M(x,\eps)) \geq \frac{3}{8} \frac{\eps^n
\Vol{B^n}}{\tVol_M (B_M(x_0,R_\theta))} \mu(A) \geq \eps^n \Vol{B^n}
f(\mu(A)) ~,
\end{equation}
where $f$ is defined as:
\[
 f(\theta) = \frac{3}{8} \frac{\theta}{\tVol_{M_{K_\theta}}(B_{M_{K_\theta}}(g(\theta/2)+1))} ~.
\]
Now let $0<s<t<1$ be close enough such that there exists an $\eps_1<\eps_t$ such that:
\begin{equation} \label{eq:control4}
t-s = \eps_1^n \Vol{B^n} f(t) ~.
\end{equation}
By definition, for any $\eta > 0$, there exists a set $A$ such that
$\mu(A) = t$ and $\mu^+(A) \leq I(t) + \eta$. By (\ref{eq:control3}) there exists an
$x\in B_M(x_0,R_t)$ such that $\mu(A \setminus B_M(x,\eps_1)) \leq s$, and since $\mu$ is absolutely
continuous, it follows that there exists an $\eps_2 \leq \eps_1$ such that $\mu(A \setminus
B_M(x,\eps_2)) = s$. Therefore:
\[
 I(s) \leq \mu^+(A \setminus B_M(x,\eps_2)) \leq \mu^+(A) + \mu^+(B_M(x,\eps_2)) \leq I(t) + \eta +
\mu_\infty(t) \frac{5}{4} \eps_1^{n-1} \Vol{S^{n-1}} ~,
\]
where we have used (\ref{eq:sa-estimate}) and (\ref{eq:control0}) in the last inequality. Sending $\eta$ to 0 and plugging in
(\ref{eq:control4}), we conclude that for some constant $C_n$ which depends on $n$:
\[
 I(s) \leq I(t) + C_n \mu_\infty(t) \brac{\frac{t-s}{f(t)}}^{\frac{n-1}{n}} ~.
\]
To get the inequality in the other direction, we require that $0<s<t<1$ are close
enough so that $\eps_1 < \eps_{1-s}$ in addition satisfies:
\[
t-s = \eps_1^n \Vol{B^n} f(1-s) ~.
\]
Now let $A \subset M$ be such that $\mu(A) = s$ and $\mu^+(A) \leq I(s) + \eta$. Applying
(\ref{eq:control3}) for the set $M \setminus A$, we find an $x \in B_M(x_0,R_{1-s})$ and $\eps_2 \leq \eps_1$ such
that $\mu(A \cup B_M(x,\eps_2)) = t$. Repeating the above argument then gives:
\[
 I(t) \leq I(s) + C_n \mu_\infty(1-s) \brac{\frac{t-s}{f(1-s)}}^{\frac{n-1}{n}} ~.
\]
Since $f$ is monotone, this concludes the proof.
\end{proof}

Our approximation argument is now clear. Given a measure $\mu$ in
the setting of Lemma \ref{lem:I-continuity}, we know that its
isoperimetric profile $I$ is continuous. Assume that $\mu$ can be
approximated from above or from within by measures $\set{\mu_m}$
satisfying our generalized smooth convexity assumptions. By Theorem
\ref{thm:Intro-concavity}, the corresponding profiles $\set{I_m}$
(and when the densities are uniform, also the renormalized profiles
$\{I_m^{n/(n-1)}\}$) are concave, and so applying Proposition
\ref{prop:convergence}, we deduce the pointwise convergence of $I_m$
to $I$, which clearly preserves concavity. We therefore deduce:

\begin{thm} \label{thm:main-approx}
Let $\Omega = (M,g)$ denote an $n$-dimensional ($n\geq 2$) smooth complete oriented connected Riemannian manifold
and let $d$ denote the induced geodesic distance. For each $m \geq 1$, let $\set{\mu_m}$ denote a sequence of
Borel probability measures on $\Omega_m \subset \Omega$ so that $(\Omega_m,d,\mu_m)$ satisfies our generalized smooth convexity assumptions.
Assume that $\set{\mu_m}$ tends to an absolutely continuous Borel
probability measure $\mu$ from above or from within,
and denote $I_m = I_{(\Omega_m,d,\mu_m)}$ and $I = I_{(\Omega,d,\mu)}$.
Then $I_m \rightarrow I$ pointwise and consequently $I$ is concave on $[0,1]$. Moreover, if each $\mu_m$ is uniform over $\Omega_m$, then $I^{n/(n-1)}$ is also concave on $[0,1]$.
\end{thm}
\begin{proof}
The argument has already been sketched. We only remark that it is not hard to verify the validity of
the assumptions of Lemma \ref{lem:I-continuity} on $\mu$, as the limit of $\set{\mu_m}$ as above (see e.g. \cite[Remark 6.2]{EMilmanRoleOfConvexityInFunctionalInqs}).
\end{proof}

\begin{cor} \label{cor:approx-bodies}
Let $\Omega$ denote any (non-smooth) convex bounded domain in
$\Real^n$ ($n \geq 2$), let $\mu$ denote the uniform probability
measure on $\Omega$ and let $d$ denote the Euclidean metric. Then
our convexity assumptions are satisfied, $I = I_{(\Omega,d,\mu)}$ is
concave on $[0,1]$, and so is $I^{n/(n-1)}$.
\end{cor}
\begin{proof}
Approximate $\Omega$ from outside by smooth convex domains using standard methods (see e.g.
\cite{Schneider-Book}). Note that $\Omega_\eps$ will only guarantee $C^1$ smoothness.
\end{proof}

\begin{cor} \label{cor:approx-lc}
Let $\Omega = \Real^n$ ($n \geq 1$), let $\mu$ denote any absolutely
continuous log-concave probability measure (with possibly non-smooth
density) and let $d=\abs{\cdot}$ denote the Euclidean metric. Then
our convexity assumptions are satisfied and $I = I_{(\Omega,d,\mu)}$
is concave on $(0,1)$ (and if $n \geq 2$, on $[0,1]$).
\end{cor}
\begin{proof}
The case $n = 1$ follows from Theorem \ref{thm:Bobkov-concavity} in the Appendix. For the case $n \geq 2$,
we will need to approximate $\mu$ from above and within by a sequence of smooth log-concave probability
measures. Since we did not find a standard reference for this, we outline the argument.

First, assume that the support $B$ of $\mu$ is compact. Approximate $\mu$ by smooth log-concave
probability measures $\set{\nu_{\eps}}$ in total-variation distance
using standard methods (e.g. convolution with a Gaussian mollifier).
Now define $\eta_{\eps,\delta}$
to be the dilatation of $\nu_{\eps}$ given by $\eta_{\eps,\delta}(A) = \nu_\eps(x_0 + (1+\delta)(A-x_0))$ for all Borel sets $A$,
where $x_0$ is a point in the interior of $B$ (another possibility would be to use sup-convolution with a small Gaussian).
It is then not hard to check that for a suitable subsequence, $\eta_{\eps,\delta(\eps)}$ tends to $\mu$ from above,
from which the assertion follows by Theorem \ref{thm:main-approx}.

In case the support of $\mu$ is not compact, we repeat the above argument for the truncated measures $\mu_r =
\mu|_{r B^n_2} / \mu(r B^n_2)$, where $B^n_2$ denotes the Euclidean unit-ball.
Note that $\mu^+(r B^n_2) \rightarrow 0$ as $r \rightarrow \infty$ by the co-area formula:
\[
 \int_0^\infty \mu^+(r B^n_2) dr = \int_0^\infty \mu^+\set{x \in \Real^n ; |x| \geq r} dr =  \int_{\Real^n} | \nabla |\cdot | \; | d\mu = 1 ~.
\]
Hence $\set{\mu_r}$ tends to $\mu$ from within, and so by Theorem
\ref{thm:main-approx} the claim now follows for arbitrary log-concave measures.
\end{proof}

\subsection{Stability of First-Moment Concentration}
\label{subsect:stability-FM}

Up to now, we have only concluded the Main Theorem \ref{thm:Main} under our \emph{smooth} convexity assumptions.
We now describe how to extend these assumptions to our general convexity assumptions.

Indeed, assume that $\mu$ can be approximated in total-variation by measures $\set{\mu_m}$ with
density $\exp(-\psi_m)$ such that $\psi_m \in C^2(M)$ and $Ric_g + Hess_g \psi_m \geq 0$ on $\Omega = (M,g)$.
We would like to show that our Main Theorem, stating that $D_{Che}(\Omega,d,\mu) \geq c D_{FM}(\Omega,d,\mu)$
for some universal constant $c>0$, still holds. It is immediate to deduce from Lemma \ref{lem:approx-easy-direction} that:
\[
 D_{Che}(\Omega,d,\mu) \geq \limsup_{m \rightarrow \infty} D_{Che}(\Omega,d,\mu_m) ~,
\]
and using our Main Theorem for the smooth measures $\mu_m$ (and Lemma \ref{lem:E-M}), we deduce that:
\[
 D_{Che}(\Omega,d,\mu) \geq c \limsup_{m \rightarrow \infty} D^M_{FM}(\Omega,d,\mu_m) ~,
\]
for some universal constant $c>0$. The First Moment constant is particularly easy to handle, since
there is no $\norm{\abs{\nabla f}}_{L_q}$ term which needs to be controlled. The following lemma, which is
an adaptation of a classical lemma of C. Borell \cite{Borell-logconcave} from the Euclidean case to the
Riemannian-manifold-with-density setting, enables us
to reduce to the case that $\set{\mu_m}$ are all supported on some compact set:

\begin{lem} \label{lem:Borell}
Let $x_0 \in M$ and $R>0$ be such that $\theta = \mu_m(B(x_0,R)) > 1/2$. Then:
\[
\forall t \geq 1\;\;\; \mu_m(M \setminus B(x_0,tR) ) \leq \theta \brac{\frac{1-\theta}{\theta}}^{\frac{t+1}{2}} ~.
\]
\end{lem}

Given this lemma, it is easy to proceed as follows. Fix $x_0 \in \Omega$ and $R>0$ so that $\mu(B(x_0,R)) \geq 3/4$.
Then for some $m_0$ and all $m \geq m_0$, we have $\mu_m(B(x_0,R)) \geq 2/3$, and hence by the lemma we conclude that:
\[
 \forall m \geq m_0 \;\; \forall t \geq 1\;\;\; \mu_m(\Omega \setminus B(x_0,tR) ) \leq 2^{-\frac{t+1}{2}} ~.
\]
Let $f_m$ denote the $1$-Lipschitz functions on $\Omega$ so that $M_{\mu_m} f_m = 0$ and $1/D^M_{FM}(\Omega,d,\mu_m) = \int |f_m | d\mu_m$ (we assume without loss of generality that the supremum is achieved). Since $f_m$ are continuous, $M_{\mu_m} f_m = 0$ and $\mu_m(B(x_0,R)) > 1/2$, there must exist a $x_m \in B(x_0,R)$ so that $f_m(x_m) = 0$. Since $f_m$ are $1$-Lipschitz, it follows that for any $t \geq 1$:
\begin{eqnarray*}
&  & \int_{\Omega \setminus B(x_0,tR)} |f_m| d\mu_m \leq \int_{\Omega \setminus B(x_0,tR)} d(x,x_m) d\mu_m(x) \\
&\leq & d(x_m,x_0) \mu_m(\Omega \setminus B(x_0,tR) ) + \int_{\Omega \setminus B(x_0,tR)} d(x,x_0) d\mu_m(x) \\
&\leq & R \brac{2^{-\frac{t+1}{2}} + \int_t^\infty 2^{-\frac{s+1}{2}} ds} ~.
\end{eqnarray*}
Hence, given $\eps>0$, there exists a $t \geq 1$ so that:
\[
\sup_{m \geq m_0} \abs{\frac{1}{D^M_{FM}(\Omega,d,\mu_m)} - \int_{B(x_0,tR)} |f_m| d\mu_m } \leq \eps ~.
\]
But since our Lipschitz functions $f_m$ are uniformly bounded on $B(x_0,tR)$ by $(t+1)R$ (by passing
through $x_m$ as before),
the convergence of $\set{\mu_m}$ to $\mu$ in total-variation implies:
\[
 \lim_{m \rightarrow \infty} \sup_{m_1 \geq m_0} \abs{\int_{B(x_0,tR)} |f_{m_1}| d\mu_m - \int_{B(x_0,tR)} |f_{m_1}| d\mu } = 0 ~.
\]
Finally, we note that for $m$ large enough, by the Markov-Chebyshev inequality (we assume here without loss of
generality that $M_\mu f_m \geq 0$):
\[
\frac{1}{2} - \frac{1}{6} \leq \mu_m\set{f_m \leq 0} - \frac{1}{6} \leq
\mu\set{f_m \leq 0} \leq \mu\set{\abs{f_m - M_\mu f_m} \geq M_\mu f_m} \leq \frac{1}{D^M_{FM}(\Omega,d,\mu) M_\mu f_m} ~,
\]
so $\abs{M_\mu f_m} \leq 3 / D^M_{FM}(\Omega,d,\mu)$. Combining everything together, we deduce that for $m$ large enough:
\begin{eqnarray*}
 \frac{1}{D^M_{FM}(\Omega,d,\mu_m)} \leq \eps + \int_{B(x_0,tR)} |f_m| d\mu_m \leq 2\eps + \int_{B(x_0,tR)} |f_m| d\mu \\
\leq 2\eps + \abs{M_\mu f_m} + \int_\Omega |f_m - M_{\mu} f_m | d\mu \leq 2\eps + \frac{4}{D^M_{FM}(\Omega,d,\mu)} ~.
\end{eqnarray*}
Since $\eps>0$ was arbitrary, we conclude that:
\[
 D_{Che}(\Omega,d,\mu) \geq c \limsup_{m \rightarrow \infty} D^M_{FM}(\Omega,d,\mu_m) \geq \frac{c}{4} D^M_{FM}(\Omega,d,\mu) ~.
\]
This concludes the proof, since as usual, we may pass from $D^M_{FM}$ to $D_{FM}$ using Lemma \ref{lem:E-M}.

\medskip

For completeness, we provide a proof of Lemma \ref{lem:Borell},
using the following remarkable generalization of the Pr\'ekopa-Leindler inequality (e.g. \cite{BrascampLiebPLandLambda1})
due to Cordero-Erausquin, McCann and Schmuckenschl{\"a}ger \cite{CMSManifoldWithDensity} (generalizing their
own result from \cite{CMSInventiones}). Given $x,y \in M$ and $s \in [0,1]$, define:
\[
 Z_s(x,y) := \set{z \in M ; d(x,z) = s d(x,y) \text{ and } d(z,y) = (1-s)d(x,y)} ~.
\]

\begin{thm}[Cordero-Erausquin--McCann--Schmuckenschl{\"a}ger] \label{thm:CMS-manifold-with-density}
Assume that $d\mu = \exp(-\psi) dvol_M$ with $\psi \in C^2(M)$ and $Ric_g + Hess_g \psi \geq 0$ on $M$.
Let $s \in [0,1]$ and $f,g,h : M \rightarrow \Real_+$ be such that:
\[
 \forall x,y \in M \;\;\; \forall z \in Z_s(x,y) \;\;\;\;\; h(z) \geq f^{1-s}(x) g^s(y) ~.
\]
Then:
\[
 \int_M h d\mu \geq \brac{\int_M f d\mu}^{1-s} \brac{\int_M g d\mu}^{s} ~.
\]
\end{thm}

\begin{proof}[Proof of Lemma \ref{lem:Borell}]
Let $t\geq 1$, and observe that:
\begin{equation} \label{eq:Borel-trick}
\forall x \in B(x_0,R) \;,\; \forall y \in M \setminus B(x_0,tR) \;\;\;\;\; Z_{\frac{2}{t+1}}(x,y) \cap B(x_0,R) = \emptyset ~.
\end{equation}
Indeed, if this is not so, there would exist a $z \in M$ so that:
\[
 d(x,z) = \frac{2}{t+1} d(x,y) \;,\; d(z,y) = \frac{t-1}{t+1} d(x,y) \;,\; d(z,x_0) < R ~.
\]
But then:
\[
 d(y,x_0) \leq d(y,z) + d(z,x_0) < \frac{t-1}{t+1} ( d(x,x_0) + d(x_0,y) ) + R < \frac{t-1}{t+1} d(y,x_0) + \frac{2t}{t+1} R ~,
\]
which would imply that $d(y,x_0) < tR$, a contradiction. Hence, (\ref{eq:Borel-trick}) implies that the
functions $f = \chi_{B(x_0,R)}$, $g = \chi_{M \setminus B(x_0,tR)}$ and $h = \chi_{M \setminus B(x_0,R)}$ satisfy
the assumption of Theorem \ref{thm:CMS-manifold-with-density} with $s = \frac{2}{t+1}$.
Theorem \ref{thm:CMS-manifold-with-density} then implies that:
\[
1-\theta \geq \theta^{\frac{t-1}{t+1}} \mu_m(M \setminus B(x_0,tR) )^{\frac{2}{t+1}} ~,
\]
and the conclusion of the lemma follows.
\end{proof}

\section*{Appendix}
\renewcommand{\thesection}{A}
\setcounter{thm}{0}
\setcounter{equation}{0}  \setcounter{subsection}{0}

In the Appendix, we provide more details regarding the statement and
ideas underlying the proof of Theorem \ref{thm:Intro-concavity} from
the Introduction, as it plays an essential role in our argument. In
the statement of this theorem, we have summarized a series of
results in Riemannian Geometry concerning the concavity of the
isoperimetric profile, which were proved under increasingly general
convexity assumptions.
An essential ingredient in the proofs of these results is provided by Geometric Measure Theory,
which guarantees the existence and regularity of the isoperimetric minimizers, and permits the use
of a variational argument to deduce the concavity of the profile.

\subsection{Manifolds with uniform densities}

First, we survey the case where the metric space $(\Omega,d)$ is given by a bounded domain
(connected open set) with $C^2$ boundary in a smooth complete oriented connected $n$-dimensional ($n\geq 2$)
Riemannian manifold $(M,g)$ along with the induced geodesic distance $d$ in $M$, and the probability
measure $\mu$ is given by the restriction to $\Omega$ of the Riemannian volume form $vol_M$ on $M$,
normalized so that $\mu(\Omega) = 1$.
We summarize for completeness some remarkable results provided by
Geometric Measure Theory about the existence and regularity of isoperimetric minimizers in the case
we are considering,
and refer to the books of Federer \cite{FedererBook}, Morgan \cite{MorganBook},
Giusti \cite{GiustiBook} and Burago and Zalgaller
\cite{BuragoZalgallerBook} for further information.

\begin{thm*}[Almgren \cite{AlmgrenExistenceAndRegularity,AlmgrenMemoirs}, Bombieri
\cite{BombieriRegularityTheory}, Gonzales--Massari--Tamanini \cite{GMT}, Gr\"{u}ter \cite{Gruter},
Morgan \cite{MorganRegularityOfMinimizers}]
For any $t\in(0,1)$, there exists an open isoperimetric minimizer $A$ of measure $t$ for the
isoperimetric problem on $(\Omega,d,\mu)$ as above.
The boundary $\Sigma = \overline{\partial A \cap \Omega}$ can be written as a disjoint union of a
regular part $\Sigma_r$ and a set of singularities $\Sigma_s$, with the following properties:
\begin{itemize}
\item
$\Sigma_r \cap \Omega$ is a smooth, embedded hypersurface of constant mean curvature.
\item
$\Sigma_r$ meets $\partial \Omega$ orthogonally.
\item
$\Sigma_s$ is a closed set of Hausdorff co-dimension not smaller than 8. This result is sharp.
\end{itemize}
\end{thm*}

For all the results to be described, it is essential that the
Hausdorff co-dimension of the singular part of the boundary is large
(although typically knowing that it is greater than 3 is
sufficient). This approach was used by M. Gromov in his influential
generalization of P. L\'evy's isoperimetric inequality
\cite{GromovGeneralizationOfLevy},\cite[Appendix C]{Gromov}. The negligible singular part permits to
consider a normal variation of the regular part, and from there on
one may continue by using the readily available tools from
Riemannian Geometry to calculate the first and second variations of volume and area.
Before proceeding, we remark that most results
we will mention deduce that the isoperimetric profile satisfies a second order differential inequality
under more general convexity assumptions than stated
(e.g. a negative lower bound on the Ricci curvature), and provide a characterization of the equality case as well.

\medskip

The first convexity assumption which we add is that the Ricci
curvature tensor $Ric_g$ of $(M,g)$ be non-negative. When $M$ is a closed
manifold and $\Omega = M$, and under the additional assumption that
all isoperimetric minimizers are smooth submanifolds (this is
always the case when $n \leq 7$), it was shown by Bavard and Pansu
\cite{BavardPansu} that $I$ is concave on $[0,1]$. In fact, these
authors attribute the same statement without the assumption on the
smoothness of the isoperimetric minimizers to B\'erard, Besson and
Gallot. This was also formally verified by Morgan and Johnson
\cite[Section 2.1 and Proposition 3.3]{MorganJohnson}. Gallot in
\cite[Corollary 6.6]{GallotIsoperimetricInqs} showed that in fact
the renormalized profile $I^{n/(n-1)}$ is concave in this case. This
result captures the right dependence of the dimension in the
exponent.

For our applications, the case where $\Omega$ is a proper subset of
$M$ is of most interest. In that case, to deduce the concavity of
the isoperimetric profile, clearly one has to add some additional
assumptions on $\Omega$. When $(M,g)$ is the Euclidean space
$(\Real^n,\abs{\cdot})$, it was first shown by Sternberg and Zumbrun
\cite{SternbergZumbrun} that a natural condition is that $\Omega$ be
convex, in which case they showed that the profile $I$ is indeed
concave. This result was further strengthened by Kuwert
\cite{Kuwert}, who showed that the renormalized profile
$I^{n/(n-1)}$ is also concave. This was then generalized by Bayle
and Rosales \cite{BayleRosales} to the case of a Riemannian manifold
with non-negative Ricci curvature, under the assumption that
$\Omega$ is \emph{locally convex}:

\begin{dfn*}
A domain $\Omega \subset (M,g)$ is said to be locally convex, if all geodesics in $M$ tangent
to $\partial \Omega$ are locally outside of $\Omega$. By a result of Bishop
\cite{BishopInBayleRosales}, in case that $\Omega$ has $C^2$ boundary, this is equivalent to
requiring that the second fundamental form of $\partial \Omega$ with respect to the normal pointing
into $\Omega$ be positive semi-definite on all of $\partial \Omega$.
\end{dfn*}

We summarize the above results in the following:

\begin{thm}[Bavard--Pansu, B\'erard--Besson--Gallot, Gallot, Morgan--Johnson,
Sternberg--Zumbrun, Kuwert, Bayle--Rosales]
\label{thm:uniform-concavity} Let $(M,g)$ be a smooth complete oriented connected
Riemannian manifold of dimension $n \geq 2$ with non-negative Ricci
curvature, and let $\Omega$ denote a locally convex bounded domain
in $(M,g)$. Let $d$ denote the induced geodesic distance in $(M,g)$
and $\mu$ the restriction to $\Omega$ of the canonical volume form
$vol_M$ on $M$, normalized so that $\mu(\Omega) = 1$. Assume in
addition that $\Omega$ has $C^2$ smooth boundary. Then the
isoperimetric profile $I = I_{(\Omega,d,\mu)}$ is a concave function
on $[0,1]$. Moreover, so is $I^{n/(n-1)}$.
\end{thm}

\subsection{Manifolds with densities}

As before, let $(M,g)$ denote an $n$-dimensional ($n\geq 2$)
smooth complete oriented connected Riemannian manifold with induced geodesic distance $d$. In
addition, let $\psi \in C^2(M)$ be such that $d\mu =
\exp(-\psi) d vol_M$ is a probability measure on $M$. Since the
influential work of Bakry and \'Emery \cite{BakryEmery} in the
abstract framework of diffusion generators, it is known that a
natural convexity condition on a manifold with density, which
replaces the condition $Ric_g \geq 0$ in the uniform density case,
is to require the following $CD(0,\infty)$ Curvature-Dimension
condition:
\begin{equation} \label{eq:RicHess}
Ric_g + Hess_g \psi \geq 0 \;\; \text{ as 2-tensor fields } ~.
\end{equation}

\begin{thm}[Bayle \cite{BayleThesis}, Morgan
\cite{MorganManifoldsWithDensity,MorganNewBook}]
\label{thm:MnfldsDensityConcave} Let $\Omega = (M,g)$ and $d,\mu$ as
above. Assume that (\ref{eq:RicHess}) holds on $\Omega$. Then $I =
I_{(\Omega,d,\mu)}$ is a concave function on $[0,1]$.
\end{thm}

This theorem was proved by Bayle in \cite{BayleThesis} under the assumption that $M$ is a closed
manifold. It was noted (without explanation) by Morgan
\cite[Corollary 9]{MorganManifoldsWithDensity}
that the same proof applies for a general complete manifold, as long as it has finite
$\mu$-measure. Indeed, Bayle's argument remains exactly the same; the only point one needs to check
is the existence and regularity of isoperimetric minimizers in the manifold with density setting.
The argument goes as follows: it was shown by Morgan in
\cite[Remark 3.10]{MorganRegularityOfMinimizers} that given a complete smooth Riemannian
manifold with positive density $\rho \in C^k(M)$ ($k \geq 0$), if there exists an area minimizing
current then its boundary is necessarily $C^k$ regular outside a set of Hausdorff codimension at
least 8. As explained e.g. in
\cite{MorganRegularityOfMinimizers,MorgansStudentThesis,MorganManifoldsWithDensity}, the existence
of an area minimizing current is guaranteed by the local compactness Theorem for currents (see
\cite{MorganBook}), as soon as the $\mu$-measure of $M$ is finite, which is always the case in our
setting. Since the minimizing
current is regular by the previous result, it follows that the usual notion of weighted area (i.e.
Minkowski boundary measure) and the weighted area of a current coincide, and hence there exists a
regular minimizer of Minkowski boundary measure.

The assumption that $M$ has finite mass is essential for the
existence of minimizers, otherwise one may construct counterexamples
(see \cite{BenjaminiCao} or \cite[p. 51]{BayleThesis}). It is also
essential that the density be continuous, otherwise minimizers need not
necessarily exist
(consider the density $\frac{1}{4} \chi_{[0,1] \times [0,1]} +
\chi_{[\frac{1}{4},1]\times[0,1]}$ on $[0,1]\times[0,1]$).

\medskip

We remark that the same existence and regularity argument works for
manifolds with a smooth boundary. Let $\Omega \subset (M,g)$ be a
domain (connected open set) with $C^2$ boundary, let $d$ be the
geodesic distance induced by $(M,g)$, and let $d\mu = \exp(-\psi) d
vol_M|_\Omega$ with $\psi \in C^2(\overline{\Omega})$ so that
$\mu(\Omega)=1$. One can easily check that the argument of
Gr\"{u}ter \cite{Gruter} on the constant curvature of the regular
part of the boundary and the orthogonality still applies, with a
minor change in the conclusion. We summarize this in the following:

\begin{thm*}[Morgan \cite{MorganRegularityOfMinimizers,MorganBook,MorganNewBook}, Gr\"{u}ter
\cite{Gruter}]
For any $t\in(0,1)$, there exists an open isoperimetric minimizer $A$ of measure $t$ for the
isoperimetric problem on $(\Omega,d,\mu)$ as above. The boundary $\Sigma = \overline{\partial A \cap
\Omega}$ can be written as a disjoint union of a regular part $\Sigma_r$ and a set of singularities
$\Sigma_s$, with the following properties:
\begin{itemize}
\item
$\Sigma_r \cap \Omega$ is a $C^2$ smooth, embedded hypersurface of constant \emph{generalized} mean
curvature, defined as:
\[
H_{\Sigma_r,\psi}(x) := H_{\Sigma_r}(x) + \frac{1}{n-1} g_x(\nabla_x \psi, \nu_{\Sigma_r}(x)),
\]
where $H_{\Sigma_r}(x)$ denotes the usual mean curvature of $\Sigma_r$ in the direction of the
unit normal $\nu_{\Sigma_r}(x)$ pointing into $A$ (i.e. the trace of the second fundamental form
divided by $(n-1)$), for $x\in \Sigma_r \cap \Omega$.
\item
$\Sigma_r$ meets $\partial \Omega$ orthogonally (even in the
presence of a density).
\item
$\Sigma_s$ is a closed set of Hausdorff co-dimension not smaller than 8.
\end{itemize}
\end{thm*}

It is then a (tedious) exercise to follow the proof of Sternberg and
Zumbrun \cite{SternbergZumbrun} and Bayle \cite{BayleThesis} (see also
\cite{BayleRosales}) and to deduce the following extension of
Theorem \ref{thm:MnfldsDensityConcave}:

\begin{thm}[after Sternberg and Zumbrun \cite{SternbergZumbrun} and Bayle \cite{BayleThesis}]
Let $\Omega \subset (M,g)$ be a locally convex domain with $C^2$
boundary, and let $d$,$\mu$ as above. Assume that (\ref{eq:RicHess})
holds on $\Omega$. Then $I = I_{(\Omega,d,\mu)}$ is a concave
function on $[0,1]$.
\end{thm}

In the one-dimensional case $n=1$, it was shown by S. Bobkov
\cite{BobkovExtremalHalfSpaces} that all of the above theorems hold
as well (here there is no point to consider a general manifold):
\begin{thm}[Bobkov] \label{thm:Bobkov-concavity}
Let $(\Omega,d) = (\Real,\abs{\cdot})$ and let $\mu$ be an arbitrary
absolutely continuous log-concave measure on $\Omega$. Then $I =
I_{(\Omega,d,\mu)}$ is a concave function on $(0,1)$.
\end{thm}
\begin{rem}
Bobkov showed that in this case, the minimizing sets are always
given by half-lines, from which it is immediate that $I(t) = \min(F'
\circ F^{-1}(t), F' \circ F^{-1}(1-t))$, where $F(s) =
\mu(-\infty,s)$. Using that $\mu$ is log-concave, direct
differentiation reveals that $I$ is concave. Note that the case
$n=1$ is special since $I$ may be discontinuous at $0$ and $1$, but
this has absolutely no consequences to our applications.
\end{rem}

\setlinespacing{0.74} \setlength{\bibspacing}{0pt}
\vspace{-20pt}

\bibliographystyle{plain}

\def\cprime{$'$}

\end{document}